\documentclass[a4paper]{amsart}

\usepackage{amsmath} 
\usepackage{amssymb}
\usepackage{amsfonts}
\usepackage{tikz}

\usepackage[latin1]{inputenc}
\usepackage{latexsym,wasysym}
\usepackage{comment}

\usepackage{graphicx}

\usepackage{pdflscape}

\usepackage{stmaryrd}   

\usepackage{enumerate}  
\usepackage{mathtools}  
\usepackage{bm}         
\usepackage{url}        

\newtheorem{theorem}{Theorem}[section]
\newtheorem{lemma}[theorem]{Lemma}
\newtheorem{proposition}[theorem]{Proposition}

\newtheorem{corollary}[theorem]{Corollary}

\theoremstyle{definition}
\newtheorem{definition}[theorem]{Definition}

\newtheorem{example}[theorem]{Example}

\newtheorem{remark}[theorem]{Remark}

\newcommand{\la}{\langle}
\newcommand{\ra}{\rangle}

\newcommand{\Al}[1][A]{\ensuremath{\mathbf{#1}} }

\newcommand{\bu}{\bullet}
\newcommand{\cb}{\circ_\mathsf{B}}
\newcommand{\ce}{\circ_\mathsf{E}}

\newcommand{\A}{\Al}

\newcommand{\QNmin}{\mathsf{QN}^{\bu}}

\newcommand{\QNminb}{\QNmin_{\mathsf{mB}}}
\newcommand{\QNbmin}{\QNmin_{\mathsf{Bm}}}

\newcommand{\PQNA}{\mathsf{PQN}}

\newcommand{\PQNmin}{\mathsf{PQN}^{\bu}}
\newcommand{\PQNminb}{\PQNmin_{\mathsf{mB}}}
\newcommand{\PQNbmin}{\PQNmin_{\mathsf{Bm}}}

\newcommand{\QNmax}{\mathsf{QN}^{\circ}}

\newcommand{\QNmaxb}{\QNmax_{\mathsf{mB}}}
\newcommand{\QNbmax}{\QNmax_{\mathsf{Bm}}}
\newcommand{\QNmaxe}{\QNmax_{\mathsf{mE}}}
\newcommand{\PQNmax}{\mathsf{PQN}^{\circ}}
\newcommand{\PQNmaxb}{\PQNmax_{\mathsf{mB}}}
\newcommand{\PQNmaxe}{\PQNmax_{\mathsf{mE}}}
\newcommand{\PQNbmax}{\PQNmax_{\mathsf{Bm}}}

\newcommand{\PQNbmaxi}{\mathsf{aiPQN}^{\circ}_{\mathsf{Bm}}}
\newcommand{\PQNmaxbi}{\mathsf{aiPQN}^{\circ}_{\mathsf{mB}}}

 \let\Omega=\varOmega

 \let\Gamma=\varGamma

 \let\Lambda=\varLambda

 \let\Sigma=\varSigma

 \let\Psi=\varPsi

 \let\Delta=\varDelta

 \let\Pi=\varPi

 \let\Theta=\varTheta

\renewcommand{\&}{\ast}

\newcommand{\nnot}{{\mathop{\sim}}}                  

\renewcommand{\nu}{\Diamond}

\let\imp\to                                

\newcommand{\alg}[1]{\mathbf{#1}}                  

 \let\equ==

\newcommand{\eval}[2][\right]{\relax
   \ifx#1\right\relax \left.\fi#2#1\rvert}

 \usepackage{wasysym}
 \usepackage{xcolor}
\usepackage{amsmath
}
\renewcommand{\phi}{\varphi}

\newcommand{\QN}{\mathsf{QN}}

\usepackage{enumerate}  
\renewcommand{\approx}{=}
\newcommand{\bb}{\bullet_\mathsf{B}}

\newcommand{\logicsf}[1]{\mathcal{ #1}}\def\d{\displaystyle}
\newcommand{\logicsfl}[1]{\mathcal{ #1}}\def\d{\displaystyle}

\newcommand{\blue}[1]{{\color{blue} #1}}
\newcommand{\red}[1]{{\color{red} #1}}

\renewcommand{\neg}{\nnot}

\newcommand{\Lf}{\ensuremath{\mathsf{L}}} %

\newcommand{\Lc}{\ensuremath{\mathcal{L}}} 

\begin{document}

\title{
Recovery operators in quasi-Nelson logic: \\ the prelinear case}

\author{
Tommaso Flaminio, Lluis Godo}
\address{IIIA-CSIC, Campus de la UAB, 08193 Bellaterra, Spain}
\email{\{tommaso, godo\}@iiia.csic.es}

\author{
 Umberto Rivieccio}
\address{Departamento de L\'ogica, Historia y Filosof\'ia de la Ciencia,  
UNED,
Madrid, Spain}
\email{umberto@fsof.uned.es}
\thanks{Tommaso Flaminio acknowledges partial support from the 2023-PUNED-0052
grant ``Investigadores tempranos UNED-SANTANDER'', the Spanish project SHORE (PID2022-141529NB-C22) funded by MCIN/AEI/10.13039/501100011033, and the H2020-MSCA-RISE-2020 project MOSAIC (Grant Agreement number 101007627).
Lluis Godo acknowledges support from the Spanish project LINEXSYS (PID2022-139835NB-C21) funded by MCIN/AEI/10.13039/501100011033.
Umberto Rivieccio acknowledges support from the 2023-PUNED-0052
grant ``Investigadores tempranos UNED-SANTANDER''
and  from the I+D+i research project PID2022-142378NB-I00 ``PHIDELO'', funded by the Ministry of Science and Innovation of Spain.}

\date{\today}




{ 



}

\begin{abstract}
This paper 
 investigates
recovery operators in 
quasi-Nelson logic, the algebraizable logical counterpart
of quasi-Nelson algebras.
These form a variety of three-potent, distributive, but not necessarily involutive
residuated lattices that may be regarded as a common generalization of
Nelson and Heyting algebras. 
We consider both consistency and
determinedness operators, with a particular focus
on logics and algebras that satisfy the prelinearity condition, which is well-known
in the area of mathematical fuzzy logics.
We show that, essentially, all algebraic and logical results already proved for (prelinear, distributive) involutive  residuated lattice-based LFIs/LFUs can be recovered 
in the quasi-Nelson setting, where one dispenses with
the involutivity assumption. In this setting, consistency and undeterminedness operators are no longer duals of one another, and hence  call for a more fine-grained algebraic and logical formalization. 
\end{abstract}

\maketitle

\section{Introduction: LFIs and LFUs}
\label{sec:intro}

\emph{Logics of formal inconsistency} (LFIs) \cite{CaCoMa,CaCo16} are among the most well-known and time-honoured 
 inconsistency-tolerant, or paraconsistent, logical systems. Formally, an LFI is usually presented as a 
{Tarskian}
(propositional) 
consequence relation $(\vdash)$ over a language 
 that may include
a conjunction $(\land)$, a disjunction $(\lor)$,
an implication $(\imp)$, truth constants $(\bot, \top)$, 
 a negation $(\nnot)$   crucially failing to satisfy the \emph{principle of explosion}
($\phi , {\nnot} \phi \vdash \bot$),  and a unary  \emph{consistency} 
operator $\circ$ (either primitive or definable from the others)
which is meant to recover a more controlled form of explosion.
This may be achieved by introducing the rule called
\emph{finite gentle principle of explosion} in~\cite[p.~50]{Taxo}:
{ 
\begin{equation}
\label{eq:fgpe}
 \phi, {\nnot} \phi, \circ \phi \vdash \bot
\end{equation}
while also requiring that there exist $\varphi$ such that
\begin{eqnarray}
\label{eq:fgpe2}
 \phi, \circ \phi \not\vdash \bot \\
\label{eq:fgpe3}  {\nnot}\phi, \circ \phi \not\vdash \bot .
\end{eqnarray}
}
The intended reading of $\circ \phi$ is ``$\phi$ is consistent'',
so~\eqref{eq:fgpe}  can be read as follows: ``if $\phi$ is consistent and contradictory, then $\phi$ explodes''.
Dually, an operator $\circ^{\delta} 
$ expressing the notion of \emph{in}consistency may be defined by
$\circ^{\delta} \phi = \nnot \circ \phi$, and the formula ``$\circ^{\delta} \phi$'' is read as ``$\phi$ is inconsistent''.


Besides 
the above-introduced white consistency operator 
and its dual, one may 
consider 
a black one ($\bu$) whose defining logical feature is the following principle: 
\begin{equation}
\label{eq:fgpd}
\vdash \phi \lor \nnot \phi \lor \bu \phi.
\end{equation}
This 
may be construed as a ``gentle excluded middle'' intended to replace the classical one 
($\vdash \phi \lor \nnot \phi$) in logics that do not satisfy this principle. 
The formula $\bu \phi$ may be read as ``$\phi$ is undetermined'' and,
 accordingly, one may paraphrase~\eqref{eq:fgpd} 
 as:
``either $\phi$ holds, or $\nnot \phi$ holds, or else $\phi$ is undetermined''. As in the preceding case, one may consider a dual connective $\bu^{\delta}$ given by
$\bu^{\delta} \phi = \nnot \bu \phi$,  reading $\bu^{\delta}  \phi$ as ``$\phi$ is determined''.

Logics that do not satisfy the classical excluded middle principle
are known as \emph{paracomplete}, and  formal systems that result from endowing a (paracomplete) logic
with an \emph{undeterminedness} connective such as $\bu$ have been called 
\emph{Logics of formal undeterminedness} (LFUs), as in~\cite{Ma05}.

Alternative principles for the black operator (that we shall not consider in the present paper)
have also appeared in the literature: for instance, one may require
$$\bu^\delta \phi \vdash \phi \lor \nnot \phi,$$
which  is indeed the defining condition for an LFU in \cite[Def.~4.5]{CaCoRo},
or  the stronger one:
$$\vdash \bu^\delta \phi \to (\phi \lor \nnot \phi).$$

A (paraconsistent and paracomplete) logic may be at the same time 
an LFI and an LFU where both 
the white and the black operator (together with their negation-duals) are simultaneously defined. 
We thus have the following 
(cf.~\cite[p.~292]{Ma05}):
\begin{align*}
\circ \quad &  \quad \text{expressing \ consistency;} \\
\circ^{\delta} = \nnot \circ \quad & \quad \text{expressing \  \emph{in}consistency;} \\
 \bu \quad & \quad \text{expressing \  \emph{un}determinedness;} \\
\bu^{\delta} = \nnot \bu \quad & \quad \text{expressing  determinedness}.
\end{align*}

{These four are collectively known as \emph{recovery operators}, or
also \emph{restoration} or \emph{perfection} operators (the terminology in the literature is not uniform,
see e.g.~\cite{CaCoRo}, and on ``perfection'' see  e.g.~\cite{GrMMR24}).

%
It may  happen, as with the logics considered in~\cite{dIRL,esteva2021}, that one operator (either the white or the black)
is sufficient for defining the remaining ones, that is, we have $\circ^{\delta} = \bu$ and $\bu^{\delta} = \circ$.
In such a case 
the four preceding notions collapse to two, with $\circ$ expressing both consistency and determinedness at the same time,
and $\bu $ expressing both inconsistency and undeterminedness. 
In the present paper 
we are going to work with a not necessarily involutive negation, suggesting that all four connectives/operators
remain distinct, though some interaction is still bound to exist (see below). 
Thus, taking both 
$\circ$ and $\bu$ 
as primitive, we shall 
consider separately the
cases where only one of them is available as well as the case where both are simultaneously present. 

The basic propositional calculi to which we shall add consistency/determinedness operators are logics arising from
 a class of residuated lattices known as \emph{quasi-Nelson algebras} (to be defined in the next section).
 In this setting, the truth-preserving (i.e., the $1$-preserving)  logics 
turn out to be explosive and, as such,
not amenable to be extended 
with a consistency operator. We shall therefore consider an alternative consequence relation canonically associated to any variety of residuated lattices, namely the \emph{degree-preserving} (or 
\emph{semilattice-based}) logic, see e.g.\ \cite{bou2009logics}. In these systems, the consequence relation 
reflects
the lattice order on the associated algebras, suggesting, firstly, 
that the rule~\eqref{eq:fgpe} translates as the equation
\begin{equation}
\label{eq:lfi}
x \land \nnot x \land \circ x = 0, \end{equation}
 and, secondly, 
 that the logic is going to be paraconsistent as long as the algebras do
not satisfy the equality $x \land \nnot x  = 0 $.
Stronger conditions naturally arise in the context of LFIs, for instance one may require
the value of 
$\circ \phi$  (which tells us whether ``$\phi$ is consistent'' or not) to be crisp, or Boolean.
These give rise to the Boolean consistency operators
$\rm{B} \!  \max $ and  $\mathsf{maxB}$, 
that we  will also discuss at length. 

The picture  changes somewhat when we consider determinedness operators. In this case, since a  residuated lattice
or quasi-Nelson algebra
need not in general satisfy the equation $x \lor \nnot x = 1$, the
associated $1$-preserving logic
turns out to be paracomplete, and may well be expanded with determinedness operators
-- of course, nothing prevents us from considering the degree-preserving logic as well.  
In either case, requiring~\eqref{eq:fgpd} will translate, on the algebraic models,
as the equality 
\begin{equation}
\label{eq:lfu}
x \lor \nnot x \lor \bu x = 1.
\end{equation}
As in the case of the white, 
we can further require the black operator to be Boolean, giving rise to what we call
$\mathsf{Bmin}$ and $\mathsf{minB}$-undeterminedness operators.

It may be interesting to note that, on algebraic models, each of the above-mentioned operators is uniquely determined
by the structure of the underlying algebra (as indeed suggested by the terms \emph{max} and \emph{min}).
This explains why non-trivial interactions arise when two different operators are simultaneously defined on the same algebra: see Section~\ref{sec:both}. 

{ 
The  paper is organized as follows. Our algebraic and logical starting point,
to  be recalled in the next section, are the variety of quasi-Nelson algebras (Subsection~\ref{intro:QN-algebras}) and quasi-Nelson logic (Subsection~\ref{prel-logics}). We shall, in particular, focus on their {\em prelinear extensions},
arising from 
those quasi-Nelson algebras that satisfy the equation
$$
(x\to y)\vee (y\to x)=1
$$
Prelinear quasi-Nelson algebras (henceforth $\mathsf{PQN}$-algebras) form a subvariety of $\mathsf{WNM}$-algebras~\cite{esteva2001monoidal} 
and have been 
studied in~\cite{rfn,FlaRi}; 
in the present paper we shall look at how recovery operators act on these 
algebras.

Undeterminedness operators on $\mathsf{PQN}$-algebras  are studied in Section~\ref{recov-undeterminded}, while consistency operators are considered in Section~\ref{prelinearity}. The latter contains one of our main results, showing that, while  semilinearity may be lost when expanding $\mathsf{PQN}$-algebras by consistency operators, it can be recovered if we postulate a quasi-equation 
forcing a suitable behaviour of the negation operator, which thus becomes ``almost involutive''.  
Although the two operators are considered separately up to Section~\ref{prelinearity}, they will be studied in the same signature for almost involutive $\mathsf{PQN}$-algebras (Section~\ref{sec:both}).  

LFIs and LFUs arising  from $\mathsf{PQN}$-algebras enriched with recovery operators
are considered in the second part of the  paper. 
Following the same 
order as in the algebraic sections,
we first deal with
logics expanded with undeterminedness connectives (Section~\ref{sec:logUndeter}) and then
with logics expanded with consistency connectives (Section~\ref{sec:logCons}). 
The latter section consists of three subsections. Subsections~\ref{sec:logCons1} and~\ref{sec:logCons2}, where we  consider the degree-preserving and the non-falsity preserving companions of prelinear quasi-Nelson  logic expanded with  recovery connectives. Lastly, in Subsection~\ref{sec:propagation}, we study how consistency and determinedness  properties propagate through propositional connectives in the case of consistency operators as well as duals of undeterminedness operators. 
The final Section~\ref{conclusions} mentions a few open problems and prospects for future work.
}

\section{Preliminaries
}
\label{sec:qnl} 

\subsection{Quasi-Nelson algebras}\label{intro:QN-algebras}
The variety of quasi-Nelson algebras ($\QN$) was introduced in~\cite{RiSp19} and further investigated in a number of subsequent publications~\cite{RiSp20,riviecciotwonegs,thiago2021negation,rivieccio2021fragments,FlaRi}.
Formally, $\QN$ can be  viewed either as a variety of 
residuated lattices or as a  generalization of  both Nelson algebras
(the algebraic counterpart of Nelson's 
logic) and Heyting algebras (see~\cite{GaJiKoOn07} for all the unexplained terminology of universal algebra, substructural logics and (residuated) lattice theory).
Taking the former approach, we may define
a quasi-Nelson algebra as a \emph{commutative, integral and bounded residuated lattice}
$\A = \la A ;  \wedge, \vee,  *, \imp, 0, 1 \ra$ that further satisfies the \emph{Nelson equation}\footnote{In the literature on (quasi-)Nelson algebras, the residuated
implication (our $\to$) is often denoted by $\Rightarrow$ in order to distinguish it from the so-called \emph{weak implication},
which is the one given, in the present notation, by the term $x^2 \to y$. In the present context, since we will not need to  employ the weak implication, we stick to the notation that is more standard in the literature on (commutative) residuated lattices.} :
\begin{equation}
\label{eq:Nelson}
\tag{Nelson}
 (x^2 \imp y) \land ((\nnot y)^2 \imp  \nnot x) \leq x \imp y 
\end{equation}
where $\nnot x : = x \imp 0$ and $x^2 = x * x$.

While~\eqref{eq:Nelson} entails 
distributivity
($x \land (y \lor z) = (x \land y) \lor (x \land z)$)
and 3-potency 
($x^3= x^2$), crucially
quasi-Nelson algebras
need not   be involutive, i.e.~they need not satisfy the double negation equation
$x =  \nnot \nnot x  $. 

The involutive 
members of $\QN$
are precisely  the Nelson algebras,
and the idempotent ones (those  satisfying $x^2 = x$) are precisely the  Heyting algebras.

\emph{Quasi-Nelson logic} $\mathcal{QN}$, the logical counterpart of $\QN$, may  be obtained by adding
the 
\emph{Nelson axiom} to $FL_{ew}$  (see~\cite{GaJiKoOn07}): 
$$
( (\phi \imp (\phi \imp \psi)) \land (\nnot \psi \imp (\nnot \psi \imp \nnot \phi)) ) \imp (\phi \imp \psi).
$$

The interest in quasi-Nelson algebras/logic is manifold, and can be motivated from a number of different perspectives 
(e.g.~constructive logics, order theory, and universal algebra; see 
the above-mentioned papers for further details).
In the present context we are mainly interested in $\QN$ as a first step in the extension of
the approach of~\cite{dIRL} (see also \cite{esteva2021}) to LFIs and LFUs beyond the involutive setting. 

Within the class of Nelson algebras (i.e., involutive quasi-Nelson algebras), the  {\em prelinearity} equation:
\begin{equation}
    (x\to y)\vee (y\to x)=1
    \tag{Pre}
\end{equation}

\noindent determines
the subvariety  generated by its totally ordered members (or {\em chains}). 
Prelinear 
Nelson algebras are otherwise known in the literature under the name of {\em nilpotent minimum algebras} ({\em $\mathsf{NM}$-algebras}), \cite{esteva2001monoidal},
and can be
equivalently defined as commutative bounded residuated lattices that are prelinear (hence they are $\mathsf{MTL}$-algebras in the language of \cite{esteva2001monoidal}), involutive, and further satisfy the equation 
$(\nnot (x\ast y) )\vee ((x\wedge y)\to (x\ast y))=1$. 

The prelinear but not necessarily involutive 
quasi-Nelson algebras 
form a variety denoted by $\PQNA$, which was  introduced and studied in~\cite{FlaRi}. This class, which is a proper subvariety of the variety of {\em weak nilpotent minimum algebras} ({\em $\mathsf{WNM}$-algebras}, see~\cite{esteva2001monoidal}), 
will play a key role in the 
present paper. 

Every prelinear subvariety of residuated lattices is generated by its totally ordered members, for the subdirectly irreducible algebras in the variety are chains
(by contrast, directly indecomposable 
prelinear residuated lattices need not be chains).
That is, prelinear subvarieties of residuated lattices are {\em semilinear}, see e.g. \cite{Horcik}. 
However,
this property 
is not necessarily preserved under language expansions;
in particular we will see (Section~\ref{prelinearity}) that, when adding a recovery operator, 
the resulting (quasi)variety may no longer be
generated by its totally ordered members. 

The following  notion will play a prominent role in the study of
LFIs based on quasi-Nelson algebras.
The set of \emph{Boolean elements  $B(\Al)$}, or complemented elements, of a bounded commutative integral residuated lattice $\Al$ may be defined as 

\begin{equation}\label{eq:boolean}
B(\A) : = \{ a \in A : a \lor \nnot a = 1\}.
\end{equation}

The set $B(\Al)$ is easily seen to be 
the universe of a sublattice of $\A$.

Another set that will play a role is the set of {\em explosive} elements of $\bf A$, defined by a condition that is, in general,
weaker 
than 
\eqref{eq:boolean}: 
\begin{equation}\label{eq:explosive}
E(\A) : = \{ a \in A : a \land \nnot a = 0\}.
\end{equation}
Indeed, it is easy to see that 
$B(\A) \subseteq E(\A)$ holds in general, for on a residuated lattice $a \lor \nnot a = 1$ entails
$a \land \nnot a = 0$, but not the other way round. However, if $\sim$ is involutive then $B(\A) = E(\A)$.


For our present concern, it is
useful 
to briefly recall from~\cite{FlaRi} how the generic algebras in $\PQNA$ can be characterized.  Consider a {\em weak-negation}  function $n:[0,1]\to[0,1]$, that is  an order reversing mapping such that $n(1) = 0$ and $n(n(x)) \geq x$ for all $x \in [0, 1]$ (see  \cite{EsDo80,cegm}). Let us further assume that $n$ satisfies the following conditions: 
\begin{enumerate}[(i)]
\item there is an element $f\in (0,1)$ that is the (unique) fixpoint, i.e., $n(f)=f$;  
\item there is a decreasing sequence $f>a_0>a_1>a_2>\ldots>0$ of points converging to $0$ such that for all $i\in \mathbb{N}$, $n(a_i)>a_i$;
\item\label{pointA}  if $0$ is a discontinuity point of $n$, 
(by the general structure of negations in fuzzy logics: see e.g., \cite{cegm}), then there exists an element $a<1$ such that $n(b)=0$ for all $b\in [a, 1]$.
\end{enumerate}
A function $n:[0,1]\to[0,1]$ as the above  behaves as shown 
in Figure~\ref{figNeg}, 
 \begin{figure}
        \centering
        \includegraphics[width=0.5\linewidth]{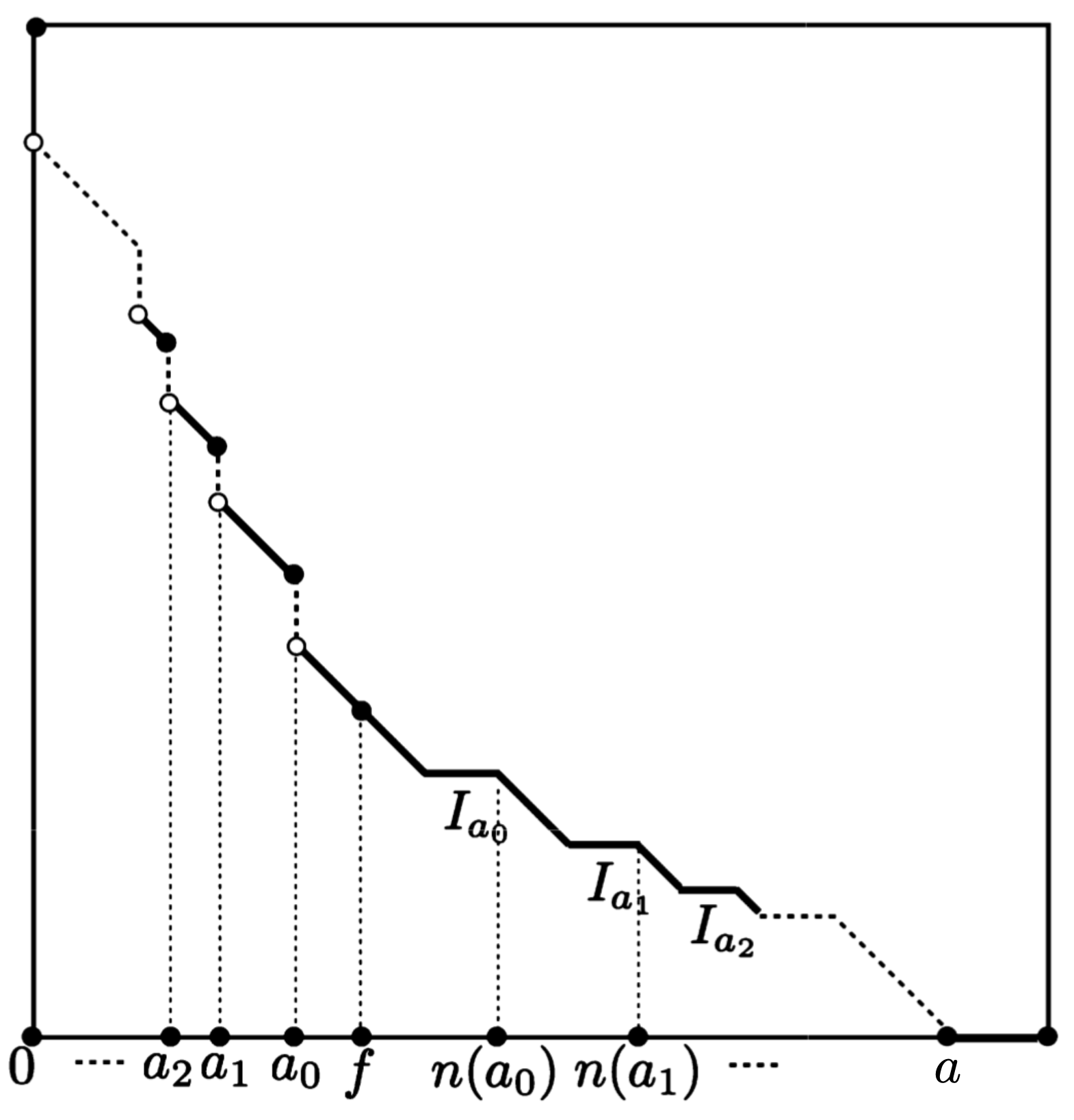}
        \caption{A negation function $n$ that determines the generic $\mathsf{PQN}$-algebra ${\bf S}_n$. Notice that the set $E({\bf S}_n)$ of explosive elements of ${\bf S}_n$ is nontrivial and coincides, in this case, with  $\{0\}\cup [a,1]$.}
        \label{figNeg}
    \end{figure} and defines a  left-continuous t-norm $\ast_n:[0,1]^2\to[0,1]$ as follows: for all $a,b\in [0,1]$,
$$
a\ast_n b=\left\{
\begin{array}{ll}
\min(a, b)&\mbox{ if }a> n(b),\\
0 &\mbox{ if }a\leq n(b).
\end{array}
\right.
$$
As explained in~\cite{cegm},
the operation $\ast_n$ thus determines a  {\em weak nilpotent minimum} algebra ({\em WNM}-algebra for short) 
that, adopting the notation of \cite{FlaRi}, we  denote by ${\bf S}_n$. For each weak negation $n$ satisftying the above conditions (i)-(iii), the algebra ${\bf S}_n$  generates the variety $\PQNA$ as proved in~\cite[Theorem 4.2]{FlaRi}.
    
    This result, together with the fact that $\PQNA$ is
    locally finite (hence
    generated by its finite chains;~\cite[Corollary~3.2]{FlaRi}) 
    shows that each ${\bf S}_n$ generates $\PQNA$. 
    We will henceforth denote by $\mathbb{S}_\infty$ the class of all these generic algebras.

\subsection{Quasi-Nelson logic}
\label{prel-logics}
The variety of (bounded, commutative, integral) residuated lattices is the equivalent algebraic semantics of the well-known substructural logic $\logicsf{FL}_{ew}$, Full Lambek calculus extended with exchange (commutativity) and weakening (integrality)~\cite{GaJiKoOn07}. An equivalent system called {\em Monoidal Logic} was previously defined and studied by H\"ohle~\cite{Hoh-95}.

The language of $\logicsf{FL}_{ew}$ consists of denumerably many propositional variables $p_1 , p_2, \ldots,$ binary connectives $\wedge, \lor, \&, \to$, and the truth constant $\bot$. Formulas, which will be denoted by lower case Greek letters $\phi, \psi, \ldots,$  are defined by induction as usual. Further connectives and constants are definable; in particular, $\nnot \phi$ stands for $\phi \to \bot$, $\top$ stands for $\nnot \bot$, and $\psi \leftrightarrow \phi$ stands for $(\psi\to \phi) \wedge (\phi \to \psi)$. A Hilbert-style calculus for $\logicsf{FL}_{ew}$ has the following
set of axioms:


\begin{itemize}
\item[(Ax1)] \quad $(\phi\to \psi)\to ((\psi \to \gamma)\to (\phi\to \gamma))$, 
\item[(Ax2)] \quad $(\gamma\to \phi)\to ((\gamma\to  \psi)\to (\gamma \to (\phi\wedge \psi)))$, 
\item[(Ax3)] \quad $(\psi\wedge \phi)\to \psi$, 
\item[(Ax4)] \quad $(\psi\wedge \phi)\to \phi$,
\item[(Ax5)] \quad $\psi\to (\psi\vee \phi)$,
\item[(Ax6)] \quad $\phi\to (\psi\vee \phi)$, 
\item[(Ax7)] \quad $(\psi \to \gamma)\to (( \phi\to \gamma)\to ((\psi\vee\phi)\to \gamma))$,
\item[(Ax8)] \quad $(\psi \& \phi)\to (\phi \& \psi )$,
\item[(Ax9)] \quad $(\psi \& \phi)\to  \psi$,
\item[(Ax10)]\quad  $(\psi\to (\phi \to \gamma))\to (( \psi\& \phi)\to \gamma)$,
\item[(Ax11)] \quad $((\psi\& \phi) \to \gamma )\to (\psi\to(\phi \to \gamma))$,
\item[(Ax12)] \quad $\bot \to \psi$,  
\item[(Ax13)] \quad $\psi\to \top$.
\end{itemize}

The only inference rule of $\logicsf{FL}_{ew}$ is modus ponens:
\begin{itemize}
\item[(MP)] \quad $\d\frac{\psi, \psi\to \phi}{\phi}$
\end{itemize}

\noindent \emph{Quasi-Nelson logic} $\mathcal{QN}$ 
is obtained as the axiomatic extension of  $\logicsf{FL}_{ew}$  with the Nelson axiom:

\begin{itemize}
\item[(Nel)] \quad $(( (\psi \& \psi)\to \phi)\wedge (({\nnot} \phi \& {\nnot} \phi) \to \nnot \psi))\to (\psi\to \phi)$.
\end{itemize}
\emph{Prelinear quasi-Nelson logic} $\mathcal{PQN}$, corresponding to the subvariety of prelinear quasi-Nelson algebras and studied in \cite{FlaRi},  is obtained by extending  $\mathcal{QN}$ by the prelinearity axiom:

\begin{itemize}
\item[(Prel)] \quad $(\phi \to \psi)\vee (\phi \to \psi )$.
\end{itemize}
As shown in \cite{FlaRi}, $\mathcal{PQN}$ can be seen  as an  axiomatic extension of  so-called {\em Weak Nilpotent Minimum} logic, that is in turn the axiomatic extension of the \em t-norm based monoidal logic}\footnote{The {\em t-norm based monoidal logic} $\logicsf{MTL}$ is the axiomatic extension of $\logicsf{FL}_{ew}$ with the prelinearity axiom (Prel), a well-known system of mathematical fuzzy logic \cite{esteva2001monoidal}.}
$\logicsf{MTL}$ by the weak nilpotent minimum axiom: 
\begin{itemize}
\item[(WNM)] \quad $(\phi\&\psi \to \bot) \vee (\phi \wedge \psi \to \phi\&\psi)$.
\end{itemize}
 
For each logic $\logicsfl{L}$ that is an axiomatic extension of $\logicsf{FL}_{ew}$, we denote by $\vdash_{\logicsfl{L}}$ the 
%
defined as usual from the corresponding sets of axioms described above and the inference rule of Modus Ponens (MP). 

As an axiomatic extension of $\logicsf{FL}_{ew}$,
each such logic $\logicsfl{L}$ is algebraizable 
and thus has an equivalent algebraic semantics given by the corresponding variety of $\logicsfl{L}$-algebras.  
This means that the truth-preserving (finitary) consequence relation $\models_{\logicsfl{L}}$ induced by the variety of $\logicsfl{L}$-algebras, defined as: \\

\noindent \begin{tabular}{lll}
$\Gamma \models_{\logicsfl{L}} \varphi$ &iff & for every $\mathsf{L}$-algebra $\alg{A}$ and every $\alg{A}$-evaluation $e$, \\
& & if $e(\psi) = 1$ for every $\psi \in \Gamma$, then $e(\varphi)=1$, 
\end{tabular}
\mbox{} \\

\noindent is such that $\vdash_{\logicsfl{L}}$ is sound and complete w.r.t.\ $\models_{\logicsfl{L}}$. 

For each such $\logicsfl{L}$ the paper~\cite{Font-GTV,bou2009logics} introduces a companion logic $\logicsfl{L}^{\mbox{\tiny $\leq$ }}$,  whose associated  consequence relation, denoted by $\models_{\logicsfl{L}}^{\mbox{\tiny $\leq$}}$, has the following semantics: for every set of formulas $\Gamma\cup\{\varphi\}$, \\

\noindent \begin{tabular}{lll}
$\Gamma \models_{\logicsfl{L}}^{\mbox{\tiny $\leq$}} \varphi$ &iff & for every $\sf{L}$-algebra $\alg{A}$, every $a \in A$, and every $\alg{A}$-evaluation $e$, \\
& & if $a \leq e(\psi)$ for every $\psi \in \Gamma$, then $a \leq e(\varphi)$.
\end{tabular}
\mbox{} 

\noindent  $\logicsfl{L}^{\mbox{\tiny $\leq$ }}$ is known as the companion logic of $\logicsfl{L}$ {\em preserving degrees of truth}, or the {\em degree-preserving companion} of $\logicsfl{L}$. 
It is not difficult to show that $\logicsfl{L}$ and $\logicsfl{L}^{\mbox{\tiny $\leq$ }}$ have the same valid formulas (i.e. $ \models_{\logicsfl{L}} \varphi$ iff $\models_{\logicsfl{L}}^{\mbox{\tiny $\leq$}} \varphi$), and that, for every finite set of formulas $\Gamma \cup \{\varphi\}$, the following property holds:
$$ \Gamma\models_{\logicsfl{L}}^{\mbox{\tiny $\leq$}} \varphi \mbox{ iff } \ \models_{\logicsfl{L}} \Gamma^\wedge \to \varphi ,$$
where $\Gamma^\wedge$ means $\gamma_1 \wedge \ldots \wedge \gamma_k$ if $\Gamma = \{\gamma_1,\ldots, \gamma_k\}$ (when $\Gamma$ is empty, then $\Gamma^\wedge$ is taken as~$\top$).

As for a syntactic presentation, if $\logicsfl{L}$ is an axiomatic extension of  $\logicsf{FL}_{ew}$, then the logic $\logicsfl{L}^{\mbox{\tiny $\leq$ }}$ admits a Hilbert-style axiomatization having  the same axioms as $\logicsfl{L}$ and the following deduction rules \cite{bou2009logics}:
\begin{description}
\item[(Adj-$\wedge$)] from $\varphi$ and $\psi$, derive $\varphi\wedge\psi$
\item[(MP-$r$)] from $\varphi$ and $\vdash_{\logicsfl{L}}\varphi\to\psi$, 
derive $\psi$.


\end{description}
Note that (MP-$r$) is a {\em restricted form} of the Modus Ponens rule, as it is only applicable when $\varphi\to\psi$ is a  theorem of $\logicsfl{L}$. Thus, although $\logicsfl{L}$ and $\logicsfl{L}^{\mbox{\tiny $\leq$ }}$ share the same theorems, $\logicsfl{L}^{\mbox{\tiny $\leq$ }}$ is in fact a weaker logic than $\logicsfl{L}$. 

If the set of theorems of $\logicsfl{L}$ is decidable, as it is the case for $\mathcal{QN}$ and $\mathcal {PQN}$, then the above system of axioms and rules provides a recursive Hilbert-style axiomatization of $\logicsfl{L}^{\mbox{\tiny $\leq$}}$. 

In the more general case of $\logicsfl{L}$ being not  an axiomatic extension but a finitary Rasiowa-implicative expansion of $\logicsf{FL}_{ew}$ (i.e.\ $\logicsfl{L}$ may have new inference rules and hence its algebraic semantics may be a sub-quasivariety of $\mathbb{RL}$), the (semantic) definition of the degree-preserving companion $\logicsfl{L}^{\mbox{\tiny $\leq$}}$ keeps being the same as above. However, the axiomatization needs to be somehow adapted. Assume the new inference rules of $\logicsfl{L}$ are: 
\begin{description}
\item[(R$_i$)] from $\Gamma_i$, derive $\varphi_i$,
\end{description}
for $i\in I$. Then, following the same idea of  \cite[Th. 2.12]{bou2009logics}, 
the following generalized result about the axiomatization of $\logicsfl{L}^{\mbox{\tiny $\leq$}}$ can be shown, see e.g.\ \cite{dIRL}.

\begin{proposition} \label{axlessequal} Let $\logicsfl{L}$ be an expansion of $\logicsf{FL}_{ew}$, with the above set of new inference rules $\{ (R_i) \}_{i\in I}$.
Then $\logicsfl{L}^{\mbox{\tiny $\leq$}}$ is axiomatized by the axioms of\/ $\logicsfl{L}$, the inference rules (Adj-$\wedge$) and (MP-$r$), and the following restricted inference rules: for each  $i\in I$,
\begin{description}
\item[(R$_i$-$r$)] from\/  $\vdash_{\logicsfl{L}} \Gamma_i$, 
derive $\varphi_i.$
\end{description}
\end{proposition}

\section{$\PQNA$-algebras with 
undeterminedness  operators}
\label{recov-undeterminded}

In this section we first recall some basic definitions from~\cite{FiRi24}  about $\mathsf{QN}$-algebras expanded with undeterminedness operators;  right after we are going to focus 
on the  expansions of $\mathsf{PQN}$-algebras with this type of operator.
%

\begin{definition}
\label{def:mincop}
A  \emph{$\mathsf{min}$-undeterminedness operator} on 
a $\mathsf{QN}$-algebra $\A$ 
is a
unary operator $\bu$ 
satisfying the following quasi-equations:
\begin{enumerate}[(i)]
    \item $x \lor \nnot x \lor y = 1$ \ implies \ $ \bu x \leq y$.
    \item $\bu x \leq y$  \ implies \ $x \lor \nnot x \lor y = 1$.
\end{enumerate}
A Boolean minimum or \emph{$\mathsf{Bmin}$-undeterminedness operator} 
is a  $\min$-undeterminedness operator 
that further satisfies the following equation
(ensuring that $\bu a$ is a Boolean element):
      \begin{enumerate}[(i)]
    \setcounter{enumi}{2}
    \item $ \bu x  \lor \nnot  \bu x = 1 $.
\end{enumerate}
\end{definition}

We shall denote by $\QNmin_{\sf m}$,
$\QNbmin$ the classes of $\mathsf{QN}$-algebras expanded by $\mathsf{min}$ and $\mathsf{Bmin}$ undeterminedness operators, respectively.  $\PQNmin_{\sf m}$ and $\PQNbmin$ will denote
the 
subclasses of the above  where prelinearity holds.

The conditions in Def.~\ref{def:mincop} can be equivalently given by the following set of equations~\cite[Thm.~1]{FiRi24}: 
\begin{enumerate}[(i)]
\item $\bu x 
 \leq y \lor \bu  (x \lor \nnot x \lor y)$.
\item  $x \lor \nnot x \lor \bu x = 1$.
\item $\bu 1 \approx 0$.
 \end{enumerate}
Therefore the classes $\QNmin_{\sf m}$, $\QNbmin$, $\PQNmin_{\sf m}$ and $\PQNbmin$
are in fact varieties.

\begin{definition}
\label{def:minBco}
A unary operator $\bb$ on a $\mathsf{PQN}$-algebra $\A$ is a  \emph{$\mathsf{minB}$-undeterminedness operator}  (minimum Boolean undeterminedness operator), if
the following (quasi-)equations are satisfied:
\begin{enumerate}[(i)]
\item $ x \lor \nnot x \lor \bb x = 1$.
\item  $\bb x \lor \nnot \bb x = 1$.
\item $ x \lor \nnot x \lor y = 1$ and $y \lor \nnot y = 1 $ imply $\bb x \leq y$.
 \end{enumerate}
\end{definition}

By $\QNminb$ we will henceforth denote the class of $\mathsf{QN}$-algebras endowed with a $\mathsf{minB}$-undeterminedness operator and $\PQNminb$ will stand for its subclasses of prelinear algebras.

Similarly to the case of $\mathsf{min}$ and $\mathsf{Bmin}$ operators, condition (iii) in Definition \ref{def:minBco} can be replaced by the following two equations \cite[Thm.~3]{FiRi24}: 

\begin{enumerate}[(i)]
\item[(iii-1)] $\bb 1 \approx 0$.
\item[(iii-2)]   $\bb x  \leq y \lor \bb (y \lor \nnot y) \lor \bb  (x \lor \nnot x \lor y)$.
 \end{enumerate}
Hence, $\QNminb$ and $\PQNminb$ are varieties as well.


Let us now focus 
on the variety of prelinear algebras, $\mathsf{PQN}$. This class 
is semilinear, i.e.~generated by its totally ordered members. Following \cite{dIRL}, we shall try to determine whether semilinearity is preserved once  recovery operators are added to the language.  
Concerning  the (quasi)varieties introduced above, let us start by observing the following facts that will serve as motivation for the results of both the present and the next section.

\begin{remark}\label{remMaxNonSemilinear}
In~\cite{dIRL} one may find 
an example of a finite $\mathsf{NM}$-algebra (i.e., an involutive $\mathsf{PQN}$-algebra: see 
\cite{esteva2001monoidal,rfn}) endowed with a $\mathsf{min}$ and a $\mathsf{minB}$  operator $\bu$ that is,
in both cases, subdirectly irreducible but not
totally ordered.\footnote{The example from \cite{dIRL} actually concerns an $\mathsf{NM}$-algebra endowed with a $\mathsf{max}$ and a $\mathsf{maxB}$ consistency operator $\circ$, see next Section \ref{prelinearity}. However, in the involutive case 
$\bu$ is definable as $\nnot \circ$,
hence the example can be applied to $\bu$ as well.} 
This entails that 
the varieties obtained as expansions of $\mathsf{NM}$ by either a $\mathsf{min}$, a $\mathsf{minB}$ 
undeterminedness operator -- 
which are proper subvarieties of (resp.) $\PQNmin_{\sf m}$, $\PQNminb$ -- 
are not generated by their totally ordered members and hence they are not semilinear. In consequence,  $\PQNmin_{\sf m}$, $\PQNminb$ 
are not semilinear either.  \hfill $\Box$
\end{remark}

The above remark 
leaves open the question whether $\PQNbmin$ is semilinear, that we will now address.  
Let us start by proving the following.
\begin{lemma}\label{lemmaBA}
    Let 
     $\langle {\bf A}, \bullet\rangle \in \PQNbmin$. Then $B({\bf A})=\{0,1\}$ if and only if ${\bf A}$ is totally ordered.
\end{lemma}
\begin{proof}
The 
right-to-left
direction is clear. To prove that $B({\bf A}) = \{0, 1\}$ implies that ${\bf A}$ is a chain, assume  by way of contradiction that it is not.
Then there are elements $a,b \in A$ such that 
$a \imp b \neq 1 \neq b \imp a$. Let $c : = a \imp b $ and $d: = b \imp a$. By the prelinearity equation,
we have $c \lor d = 1$.
If we had $c = 0$, then we would have
$c \lor d = d = b \to a = 1$, against our assumptions. 
A symmetric reasoning applies to $d$, so we have
$c, d > 0 $, therefore  $c, d \notin B(\A)$.
Therefore, $c\vee \nnot c\vee d=1$ as well and hence, by Definition \ref{def:mincop}, $\bullet c\leq d<1$. Since $\bullet c$ is Boolean, necessarily $\bullet c=0$, hence $c\in B({\bf A})$ by \cite[Proposition 1 (i)]{FiRi24},  contradicting the assumption that $c,d\not\in B({\bf A})$.
\end{proof}

The preceding lemma gives us the following.

\begin{theorem}\label{thm:PQNminSemiLin}
    The variety  $\PQNbmin$ is semilinear. 
\end{theorem}
\begin{proof}
By \cite[Theorem 5]{FiRi24} (see also \cite[Theorem 3.8]{dIRL})  the algebras 
whose Boolean elements are in $\{0,1\}$ are exactly the 
 directly indecomposable
members of $\PQNbmin$. Thus, by Lemma \ref{lemmaBA} every directly indecomposable (a fortiori, every subdirectly irreducible)  algebra in $\PQNbmin$ is totally ordered; hence $\PQNbmin$ is semilinear. 
\end{proof}

As a corollary of the above result we obtain that $\PQNbmin$ is a semisimple variety.


\begin{corollary}
Let $\langle {\bf A}, \bullet\rangle \in \PQNbmin$. 
The following conditions are equivalent:
\begin{enumerate}[(i)]
\item $\langle\Al, \bu\rangle $ is simple.
\item $\langle\Al, \bu\rangle $ is subdirectly irreducible.
\item $\langle\Al, \bu\rangle $ directly indecomposable.
\item  $B(\Al) = \{0, 1 \}$.
\item $\A$ is totally ordered.
\end{enumerate}
    Therefore, $\PQNbmin$ is semisimple.
\end{corollary}
\begin{proof}
A simple adaptation of the proof of \cite[Theorem 3.16]{dIRL} shows that the claims (i)--(iv) are all equivalent 
and by Theorem~\ref{thm:PQNminSemiLin} (iv) implies (v). Thus, it is enough to prove that (v) implies (iv). The claim is indeed trivial because the unique Boolean chain is $\{0,1\}$. The semisemplicity of $\PQNbmin$ hence directly follows.
\end{proof}
Let us notice that on a $\mathsf{PQN}$-chain $\Al$ there is only one way to define a $\mathsf{min}$-undeterminedness operator. Indeed, recall that $a \vee \nnot a <1$ for all $a \in \Al$ with $a \not\in\{0,1\}$, therefore the minimum  
element $b$ of $\Al$ such that $a\vee\nnot a\vee b=1$ is necessarily $1$. Thus, on a chain $\Al$, such an operator $\hat\bu$ is given, for all
$a \in A$, by:
\begin{equation}\label{eqBMinChain}
\hat\bu a=\left\{
\begin{array}{ll}
0&\mbox{ if }a\in  \{0,1 \} 
\\
1&\mbox{ otherwise.}
\end{array}
\right.
\end{equation}
Notice that the operator $\hat\bu$ is also a $\mathsf{Bmin}$ and a $\mathsf{minB}$-undeterminedness operator

We now want to show that, in addition to being semilinear and semisimple, the variety $\PQNbmin$ is single chain generated, that is to say,  the whole variety $\PQNbmin$ can be generated by a single standard chain. 
To this end let us first prove the next easy result.
\begin{proposition}\label{prop:locFin}
Let $\langle \Al, \hat\bu\rangle$ be any chain in $\PQNbmin$ and let $X$ be a finite subset of $A$. Then the 
subalgebra generated by $X$ is finite. 
\end{proposition}
\begin{proof}
Every finitely generated $\mathsf{PQN}$-chain is finite (cf.~\cite[Cor.~3.2.]{FlaRi}). Therefore, if $X$ is a finite subset of a $\mathsf{PQN}$-chain $\Al$, the subchain of $\Al$ generated by $X$ is finite as $\mathsf{PQN}$-algebra, and the unique $\hat\bu$ definable  is that of $\Al$ restricted to the new domain. Thus the claim follows.
\end{proof}

Now, recall the class $\mathbb{S}_\infty$ of the generic algebras in $\PQNA$  described in Section \ref{intro:QN-algebras}.  
Take any  ${\bf S}_n\in \mathbb{S}_\infty$ and let us consider the algebra $\langle {\bf S}_n, \hat\bu \rangle$ 
where  $\hat\bu$  is  necessarily defined  as in (\ref{eqBMinChain}) above and it is the  unique undeterminedness operator on ${\bf S}_n$.  
The same proof strategy  used in \cite[Theorem 4.2]{FlaRi} to prove that ${\bf S}_n$ generates the variety $\PQNA$ trivially extends to this case.
Therefore the following holds.
\begin{theorem} \label{Sinfty}
For every $\mathsf{PQN}$-chain ${\bf S}_n$ that is generic for $\PQNA$,  the algebra $\langle{\bf S}_n,\hat\bu\rangle$ generates the  variety $\PQNbmin$. 
\end{theorem}

Let us end this section observing that although $\PQNmin$ and $\PQNminb$ are not semilinear, as shown in Remark~\ref{remMaxNonSemilinear},  a way to endow $\mathsf{PQN}$-algebras with $\mathsf{min}$ or $\mathsf{minB}$ undeterminedness operators, while keeping semilinearity, 
is to consider the sub-quasivarieties of $\PQNmin$ and $\PQNminb$ obtained by imposing  the following quasi-equation:
\begin{equation}
\text{ if } (x \leftrightarrow y)\lor z = 1 , \text{ then } (\bu x \leftrightarrow \bu y)\lor z = 1
\tag{Con-$\vee^\bu$}
\end{equation}

Let us denote them by $\PQNmin\mbox{-}\vee$ and $\PQNminb\mbox{-}\vee$ respectively. By 
 \cite[Corollary 6.2.8]{Cintula-Noguera-book}, the (algebraizable) logics whose equivalent algebraic semantics are $\PQNmin\mbox{-}\vee$ and $\PQNminb\mbox{-}\vee$, respectively, are semilinear. Thus, precisely $\PQNmin\mbox{-}\vee$ and $\PQNminb\mbox{-}\vee$ are  generated by their totally  members. On the other hand, we have seen that the unique undeterminedness operator that can be defined on a $\mathsf{PQN}$-chain is that of (\ref{eqBMinChain}). Hence, the chains in $\PQNbmin$, $\PQNmin\mbox{-}\vee$ and $\PQNminb\mbox{-}\vee$ coincide. 
Now, since it is generally true that 
$\mathsf{Q}(\mathsf{K}) \subseteq \mathsf{V}(\mathsf{K}) $
for any class of algebras $\mathsf{K}$, 
we have the next result.

\begin{proposition}\label{Prop:Var-V}
$\PQNbmin = \PQNmin\mbox{-}\vee = \PQNminb\mbox{-}\vee$.
\end{proposition}

\section{$\PQNA$-algebras 
with 
consistency operators
} \label{prelinearity}


The case of 
the consistency operators in the setting of $\mathsf{PQN}$-algebras generalizes what was studied in~\cite{dIRL}, which considered the expansion of involutive $\mathsf{MTL}$-algebras ($\mathsf{IMTL}$-algebras) by consistency operators $\circ$. Indeed, on $\mathsf{IMTL}$-algebras the $\circ$ and $\bu$ operators are  interdefinable, which is no longer the case in $\mathsf{PQN}$-algebras;  hence 
the study of the expansion of $\mathsf{PQN}$-algebras with consistency operators is not implicitly  covered by~\cite{FiRi24}.


\begin{definition}
\label{def:maxcop}
A  \emph{$\mathsf{max}$-consistency operator} on a $\mathsf{QN}$-algebra $\A$ 
is a unary operator $\circ$ that satisfies the following quasi-equations: 
$$
 x \land \nnot x \land y = 0 \qquad \text{ if and only if } \qquad  y \leq \circ x.
$$
A \emph{$\mathsf{Bmax}$-consistency operator} 
is a  $\mathsf{max}$-consistency operator that further satisfies the equation
$ \circ x  \lor \nnot  \circ x = 1 $ (ensuring that $\circ a$ is a Boolean element, for each $a \in A$).
\end{definition}
 By \cite[Prop. 4.3]{ConigBook}, the quasi-equation involved in the preceding definition of a max-consistency operator can be 
equivalently replaced by the following equations: 
\begin{center}
(i) $y \land \circ  (x \land \nnot x \land y) \leq \circ x$; \qquad (ii) $x \land \nnot x \land \circ x = 0$; \qquad (iii) $\circ 0 \approx 1$.
\end{center}
Thus the classes of $\mathsf{QN}$-algebras expanded by either a $\mathsf{max}$ or a $\mathsf{Bmax}$ operators are varieties. 
Following the earlier notational convention, we will respectively denote by $\QNmax_{\sf m}$ and $\QNbmax$ 
the varieties 
of $\mathsf{QN}$-algebras expanded by a $\mathsf{max}$ 
 and a $\mathsf{Bmax}$ consistency operators, while $\PQNmax_{\sf m}$ and $\PQNbmax$ will denote 
 the expansions of their prelinear sub-varieties.

As in the involutive case and in the case of undeterminedness operators considered earlier,
we shall also study a $\mathsf{maxB}$-consistency operator, defined as follows.

\begin{definition}
\label{def:maxBco}
A unary operator $\cb$ on a quasi-Nelson algebra $\A$ is a  \emph{$\mathsf{maxB}$-consistency operator} if:
\begin{enumerate}[(i)]
\item $ x \land \nnot x \land \cb x = 0$.
\item  $\cb x \lor \nnot \cb x = 1$.
\item $ x \land \nnot x \land y = 0$ and $y \lor \nnot y = 1 $ imply $y \leq \cb x$.
 \end{enumerate}
\end{definition}
By $\QNmaxb$  and $\PQNmaxb$ we shall respectively denote the quasivarieties of $\mathsf{QN}$-algebras and $\mathsf{PQN}$-algebras with  $\mathsf{maxB}$ consistency operators.

\begin{remark}
\label{rem:cecb}
In~\cite{ConigBook} 
yet another kind of operator was introduced, which we denoted by $\ce$ and dubbed a \emph{$\mathsf{maxE}$ operator}.
The letter E refers to the requirement that, for each element $a $ in a quasi-Nelson algebra $\A$, the element $\ce a$
be not a Boolean but an explosive element, i.e., that $\ce a \in E(\A)$.
This amounts to replacing the second and third items in Definition~\ref{def:maxBco} by the following two:
\begin{enumerate}[(i')]
      \setcounter{enumi}{1}

    \item  $\ce x \land \nnot \ce x = 0$.
\item $ x \land \nnot x \land y = 0$ and $y \land \nnot y = 0$ imply $y \leq \ce x$.
\end{enumerate}
This latter definition is
perhaps more suitable for consistency operators,
and is certainly
  smoother from a technical point of view:
in particular, as pointed out in~\cite{ConigBook}, 
it can be easily rephrased in purely equational terms. 
(By contrast, we do not currently know whether the class $\QNmaxb$ is equationally definable.)
The picture, however, simplifies in the prelinear case, for one can show that, under the prelinearity assumption,
the two definitions are equivalent,
i.e., $\ce = \cb$ (entailing, in particular,  that the class $\PQNmaxb$ is also equational). 
We shall therefore no longer consider {$\mathsf{maxE}$ operators}
in the present setting, but we refer the reader to~\cite{ConigBook} for further details,
including  necessary and sufficient conditions for having $\ce = \cb$.
\end{remark}

Observe that the same argument presented in Remark~\ref{remMaxNonSemilinear} for $\bu$ also  applies to the cases of $\PQNmax$ and $\PQNmaxb$.   
%
In other words, 
the varieties obtained as expansions of $\mathsf{NM}$ by either a $\mathsf{max}$, or a $\mathsf{maxB}$  consistency operator -- 
which are proper subvarieties of (resp.) $\PQNmax_{\sf m}$ and $\PQNmaxb$ -- 
are not generated by their totally ordered members and hence  are not semilinear. In consequence,  $\PQNmax_{\sf m}$ and $\PQNmaxb$ are not semilinear either.

%
\begin{figure}
\begin{center}
 \begin{minipage}{4cm}
\includegraphics[width=3cm]{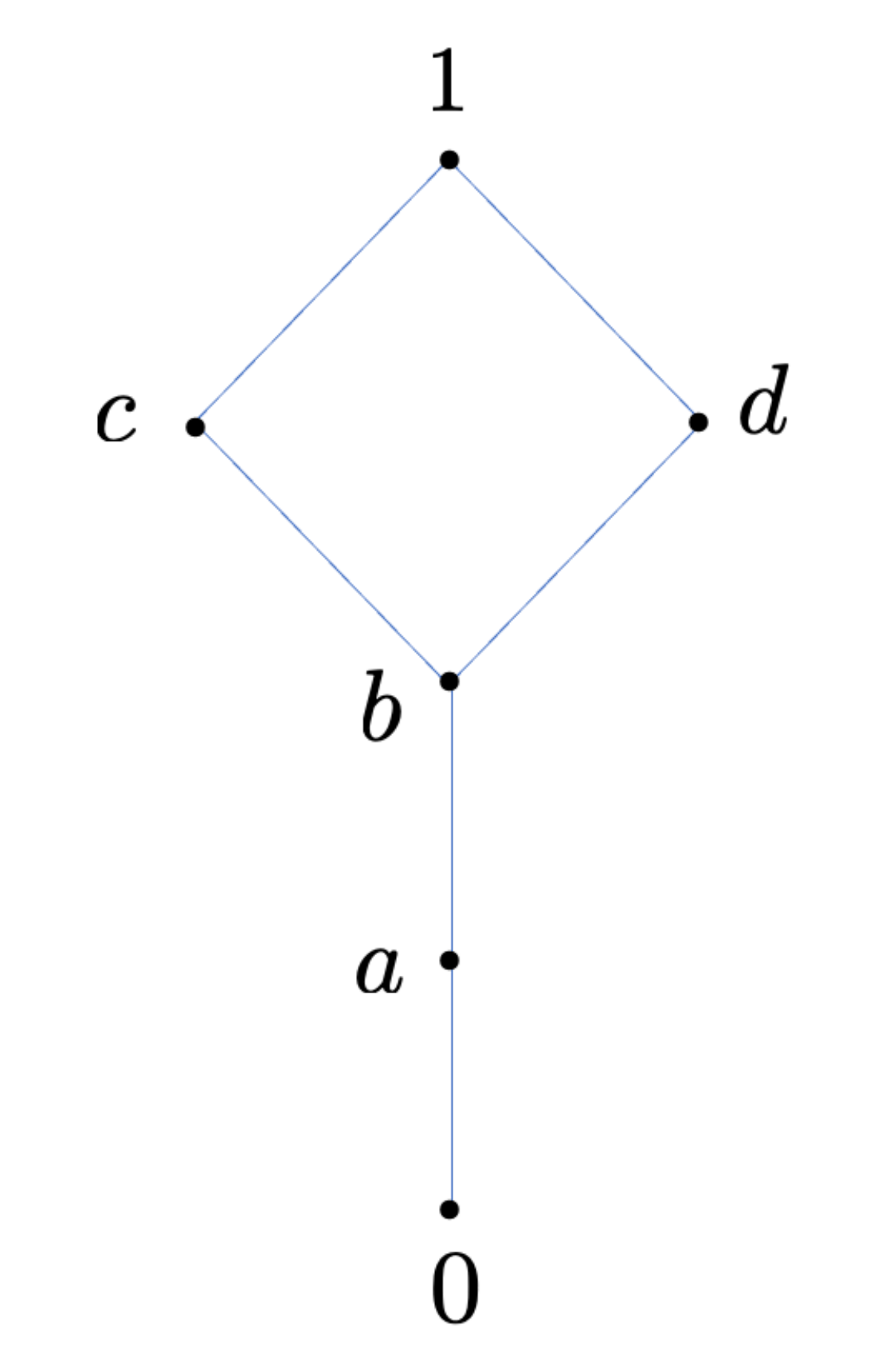}
\end{minipage}
\begin{minipage}{3cm}
\begin{tabular}{c||c|c}
 & $\neg$ & $\circ$  \\
\hline\hline
0 & 1 & 1 \\
\hline
a & d  & 0  \\
\hline
b & a  & 0  \\
\hline
c & 0 & 1  \\
\hline
d & a & 0  \\
\hline
1 & 0 & 1  \\
\end{tabular} 
\end{minipage}
\end{center}
\vspace{0,3cm}

\begin{tabular}{c||c|c|c|c|c|c}
$*$ & 0 & a & b & c & d & 1 \\
\hline\hline
0 & 0 & 0 & 0 & 0 & 0 & 0 \\
\hline
a & 0 & 0 & 0 & a & 0 & a \\
\hline
b & 0 & 0 & b & b & b & b \\
\hline
c & 0 & a & b & c & b & c \\
\hline
d & 0 & 0 &b & b & d & d \\
\hline
1 & 0 & a & b & c & d & 1 \\
\end{tabular}
\hspace{1cm}
\begin{tabular}{c||c|c|c|c|c|c}
$\to$ & 0 & a & b & c & d & 1 \\
\hline\hline
0 & 1 & 1 & 1 & 1 & 1 & 1 \\
\hline
a & d & 1 & 1 & 1 & 1 & 1 \\
\hline
b & a & a & 1 & 1 & 1 & 1 \\
\hline
c & 0 & a & d & 1 & d & 1 \\
\hline
d & a & a &c & c & 1 & 1 \\
\hline
1 & 0 & a & b & c & d & 1 \\
\end{tabular} 
\caption{Hasse diagram and operation tables of the $\PQNbmax$-algebra $({\bf A}_6,\circ)$. }\label{figCounterExample}
\end{figure}
Notice that the case of $\PQNbmax$ is not covered by the above argument and hence let us start by showing that $\PQNbmax$ is not semilinear. The proof of this fact uses the algebra whose Hasse diagram is shown in Figure~\ref{figCounterExample} that for a later use we will denote by  $\langle {\bf A}_6,\circ\rangle$.

%
In fact, $\langle{\bf A}_6,\circ\rangle$ is not totally ordered but subdirectly irreducible.  
Indeed, a direct inspection on the definition of its operations shows that the only congruences of $\langle{\bf A}_6,\circ\rangle$ are the identity and $\Theta_c=\{(x,y)\mid (x\to y)\wedge (y\to x)\geq c\}$, the latter being the unique atom of the lattice of congruences of $\langle {\bf A}_6,\circ\rangle$. Therefore,
we have that the variety $\PQNbmax$ is not semilinear, while $\PQNbmin$ is semilinear by Theorem~\ref{thm:PQNminSemiLin}.  
As it will be clear in a short while, the discrepancy between the semilinearity of $\PQNbmin$ and the failure of semilinearity for $\PQNbmax$ in part depends on the behavior of the non-involutive negation of $\mathsf{PQN}$-algebras. Indeed, by requiring the negation to be involutive, one would immediately recover the semilinearity of the corresponding variety expanded with a consistency operator $\circ$. However, involutivity is a sufficient but not necessary condition to identify a semilinear subvariety of $\PQNbmax$, as we are going to show.
%


To begin with, let us observe that outside the $\mathsf{IMTL}$-setting, where a major role in the description of the generators of the varieties $\mathsf{IMTL}^\circ$, $\mathsf{IMTL}^\circ_\mathsf{Bmax}$ and $\mathsf{IMTL}^\circ_\mathsf{maxB}$ is played by the Boolean elements $B({\bf A})$ of $\langle {\bf A}, \circ\rangle$ (see, e.g.,~\cite{CEG14}), it is convenient to consider the  explosive elements instead. 
Recall that these are defined by (\ref{eq:explosive}) as follows: 
$$
E({\bf A})=\{x\in A: x\wedge\nnot x=0\}.
$$
As we already observed, $B({\bf A})\subseteq E({\bf A})$ in general, while $E({\bf A})=B({\bf A})$ if ${\bf A}$ has an involutive negation. Compare the next result with the above Lemma \ref{lemmaBA} that is specific for the $\mathsf{Bmin}$ operator $\bu$.

\begin{lemma}\label{lemmaEA}
 Let $\langle {\bf A}, \circ\rangle$ be either in $\PQNbmax$ or $\PQNmaxb$ and
 such that $E({\bf A})=\{0,1\}$. Then ${\bf A}$ is totally ordered.
\end{lemma}
\begin{proof}
    Suppose $E({\bf A})=\{0,1\}$ and assume, by way of contradiction, that ${\bf A}$ is not totally ordered. 

    Then there are elements $a,b \in A$ such that 
$a \imp b \neq 1 \neq b \imp a$. Let $c : = a \imp b $ and $d: = b \imp a$. Thus $c \neq 1 \neq d$. By the prelinearity equation, we have $c \lor d = 1$. It follows then that both $c \neq 0$ and $d \neq 0$, since otherwise it would imply that either $c \lor d = c < 1$ (if $c = 0$) or $c \lor d = d < 1$ (if $d = 0$).  Therefore, $c, d \notin E(\A)$.

In consequence, $c\wedge\nnot c\neq 0$ and $d\wedge\nnot d \neq 0$, and hence $c, \nnot c, d, \nnot d>0$.
    Now, $c\wedge\nnot c\wedge\nnot d\leq \nnot c\wedge \nnot d=\nnot(c\vee d)=\nnot 1=0$, so $c\wedge\nnot c\wedge\nnot d=0$. By the definition of $\circ$ we have $\circ c\geq \nnot d>0$.
   Since $\circ c \in B({\bf A})\subseteq E({\bf A})=\{0,1\}$, but 
     $\circ c \neq 0$, one has that $\circ c=1$. 
    Hence, $0 = c \land \nnot c \land \circ c = c \land \nnot c$, that is,
    $c\in E({\bf A})$, contradicting what we have shown above. 
    We conclude that ${\bf A}$ is totally ordered.
\end{proof}
As we recalled above, in the involutive case $\PQNbmax$ and $\PQNmaxb$ algebras $\langle{\bf A},\circ\rangle$ are totally ordered  iff $B({\bf A})=\{0,1\}$, while the above result points out the role of $E({\bf A})$ in this more general setting. As a matter of fact,  
the following example shows a non-totally ordered algebra $\langle \Al,\circ\rangle$ in $\PQNbmax$ such that $B(\Al)=\{0,1\}$. 

    \begin{figure}
        \centering
        \includegraphics[width=0.5\linewidth]{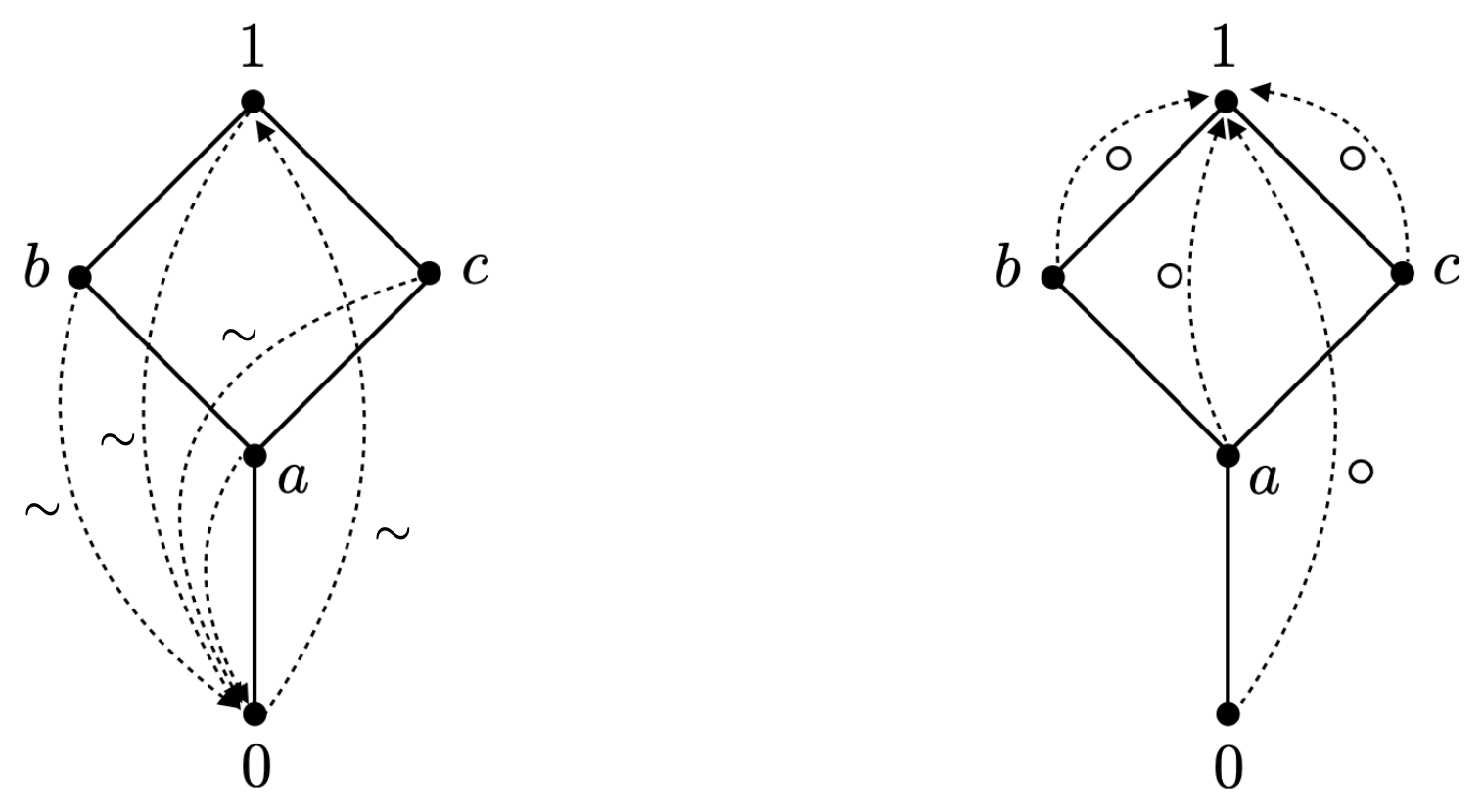}
        \caption{The directly indecomposable G\"odel algebra ${\bf G}_5$ with its negation (on the left-hand side) and the $\mathsf{Bmax}$ consistency operator considered in Example \ref{exG5}.}
        \label{figG5}
    \end{figure}
    
\begin{example}\label{exG5}
Since G\"{o}del algebras form a (proper) subvariety of $\PQNA$~\cite{FlaRi}, 
the (directly indecomposable) five-element G\"{o}del algebra ${\bf G}_5$ 
displayed in Figure~\ref{figG5} provides an example of interest.
%
As in every directly indecomposable G\"odel algebra, the negation of ${\bf G}_5$ maps to $0$ to $1$ and all the other elements 
to $0$ (see the left-hand side of Figure \ref{figG5}). Thus, $E({\bf G}_5)=G_5$. 

The operator  $\circ: G_5\to G_5$ such that
$
\circ a=1\mbox{ for all }a\in G_5
$
(right-hand side of Figure \ref{figG5}) 
is a $\mathsf{Bmax}$ consistency operator (cf.~Remark~\ref{rem:pseudocowithball}).  We note that, on a logical level, this obviously suggests that G\"odel (or Heyting) algebras
are not suitable semantic models for LFIs, for even their degree-preserving logic is explosive.

Now, $\langle {\bf G}_5, \circ\rangle$ is a non-totally ordered algebra in $\PQNbmax$ 
such that $B({\bf G}_5)=\{0,1\}$. Notice, however, that this does not provide a counterexample to
Lemma~\ref{lemmaEA},
for $E({\bf {\bf G}_5 }) \neq \{ 0, 1 \}$.
\end{example}

In general, the converse of Lemma~\ref{lemmaEA} does not hold.
Indeed, note that on any totally ordered algebra ${\bf A}$, necessarily $B({\bf A})=\{0,1\}$; however, 
even in such a case one may have
$E({\bf A})\neq \{0,1\}$. This is for instance the case of some of the standard $\mathsf{PQN}$-chains ${\bf S}_n$ studied in~\cite{FlaRi}. In fact, consider a  $\mathsf{PQN}$-algebra  on the real unit interval $[0,1]$, for which there exists an element $a\in(0,1)$ such that 
(recall Figure \ref{figNeg}):
$$
E({\bf S}_n)=\{0\}\cup[a, 1].
$$

On a $\mathsf{PQN}$-chain $\Al$ (more generally, if $0$ is meet-irreducible) there is only one way to define a max-consistency operator (that is, at the same time,  also a $\mathsf{maxB}$ and a $\mathsf{Bmax}$ operator): that is, for every $a\in A$,
\begin{eqnarray}\label{maxconschain}
\hat\circ a=\left\{
\begin{array}{ll}
1&\mbox{ if }a\in  E(\Al)\\
0&\mbox{ otherwise.}
\end{array}
\right.
\end{eqnarray}


That is, the elements that $\hat\circ$ deems consistent are precisely the explosive ones. 
Furthermore, let us  observe that the directly indecomposable members of  $\PQNbmax$ are indeed those $\langle \Al, \circ\rangle$ such that $B(\Al)=\{0,1\}$. Observe that the same proof of \cite[Theorem 3.8]{dIRL}, showing the next claim in the case of involutive algebras, works in this slightly more general case  (for the involutivity of $\Al$ is never used in the proof).
\begin{lemma}\label{LemmaDireIndBmax}
Let $\langle \Al,\circ \rangle$ be any $\mathsf{PQN}$-algebra with a consistency operator. Then $\langle \Al,\circ \rangle$ is directly indecomposable iff $B(\Al)=\{0,1\}$. 
\end{lemma}

Now we can identify a group of equivalent conditions 
that will allow us to recover  semilinearity for 
not necessarily involutive  subvarieties of $\PQNbmax$ and $\PQNmaxb$. Details will be given after the proof of the next result, that indeed holds in the more general setting of (not necessarily prelinear)  quasi-Nelson  algebras. 

\begin{lemma} \label{lem:lemexplo} For a quasi-Nelson algebra $\A$ endowed with a $\max$-consistency operator, the following conditions are equivalent:
\begin{enumerate}[(i)]
    \item\label{1} $E(\A) = B(\A)$.
    \item\label{2} $\Al$ satisfies the quasi-equation: $\nnot x = 0$ implies $x = 1$.
        \item\label{4} $\Al$ satisfies $x \lor \nnot x \lor \nnot \circ x = 1$.
\end{enumerate}
\end{lemma}
\begin{proof}
    (\ref{1})$\Leftrightarrow$(\ref{2}). Suppose (\ref{1}) holds, and assume $\nnot a = 0$ for some $a \in A$.
    Then $a \land \nnot a = 0$, so $a \in  E(\A) = B(\A)$.
    Hence, $a = a \lor \nnot a = 1$, as required. Conversely, assume (\ref{2}) holds. As observed earlier, the inclusion $B(\A) \subseteq E(\A)$ always holds on a quasi-Nelson algebra. For the other inclusion, assume $a \in E(\A)$, that is, $a \land \nnot a = 0$. Since 
    $a \land \nnot a = \nnot a \land \nnot \nnot a = \nnot (a \lor \nnot a)$, we can apply (\ref{2}) to obtain $a \lor \nnot a =1$.
    Hence, $a \in B(\A)$, as required. 


(\ref{2})$\Leftrightarrow$(\ref{4}). Assuming (\ref{2}) holds, let $a \in A$. 
We have 
$\nnot (a \lor \nnot a \lor \nnot \circ a ) 
= \nnot a \land \nnot  \nnot a \land \nnot \nnot \circ a  
=  \nnot a \land a \land \circ a  = 0$. Hence, 
applying (\ref{2}), 
we obtain 
$a \lor \nnot a \lor \nnot \circ a = 1$.
Conversely,
Let $a \in A$ be such that
$\nnot a = 0$. Then $a \land \nnot a = 0$, so $\circ a = 1$. 
By (\ref{4}), we have 
$1 = a \lor \nnot a \lor \nnot \circ a = a \lor 0 \lor \nnot 1 = a$, as required.
%
\end{proof}

\begin{remark}
\label{rem:pseudocowithball}
If a 
quasi-Nelson algebra $\Al$ satisfies the equation $x \land \nnot x = 0$, then its unique 
(Boolean) $\max$-consistency operator $\circ$ is given by $\circ a = 1$ for all $a \in A$. (The converse is also true:
$\circ x = 1$ 
entails $x \land \nnot x = 0$.) Such an algebra will also trivially satisfy the quasi-equation:
${(x \leftrightarrow y)\lor z = 1} $ implies $ {(\circ x \leftrightarrow \circ y)\lor z = 1} 
$, simply because the consequent is always true. On the other hand, it is easy on $\Al$ to falsify 
(e.g.~the first 
item of) Lemma~\ref{lem:lemexplo}: indeed, any Heyting algebra that is not Boolean must falsify it (see for instance Example~\ref{exG5}, which is built on a 5-element G\"odel algebra).
\end{remark}

The main property that will be used in the remaining part of this section is the equivalence between items
(\ref{2}) and 
(\ref{4}) of Lemma~\ref{lem:lemexplo}. To appreciate the effect of the quasi-equation in (\ref{2}) on $\mathsf{PQN}$-algebras 
(which does not mention any recovery operator), observe that if a standard $\mathsf{PQN}$-algebra ${\bf S}_n$ satisfies it, then the unique element $a\in (0,1]$ such that $n(b)=0$ for all $b\in [a,1]$ is  $a=1$. This is  the reason why $E({\bf S}_n)=B({\bf S}_n)$, as ensured by (\ref{1}) of Lemma~\ref{LemmaDireIndBmax}. Thus, the negations of these standard $\mathsf{PQN}$-algebras have an involutive behavior in a left interval of $1$. A negation of that kind is, for instance, the one on Figure \ref{figNewNeg}. Standard $\mathsf{PQN}$-algebras having those negations will be called {\em almost involutive}. 

Let us denote by $\mathsf{ai}\PQNA$ the sub-quasi-variety of $\PQNA$ 
that satisfies the quasi-equation 
in Lemma~\ref{lem:lemexplo} (ii). 
%
%
%
For algebras in  any among $\PQNmax_{\sf m}$, $\PQNmaxb$, $\PQNbmax$ and $\PQNmaxe$,
this quasi-equation is equivalent to the equation $x \lor \nnot x \lor {\nnot} {\circ} x = 1$. 
%
Thus, for any $\mathsf{V}$ among these varieties, we may consider its 
``almost involutive''
sub-variety $\mathsf{aiV}$ 
obtained by adding the  
equation in Lemma~\ref{lem:lemexplo} (iii).

\begin{figure}
        \centering
        \includegraphics[width=0.5\linewidth]{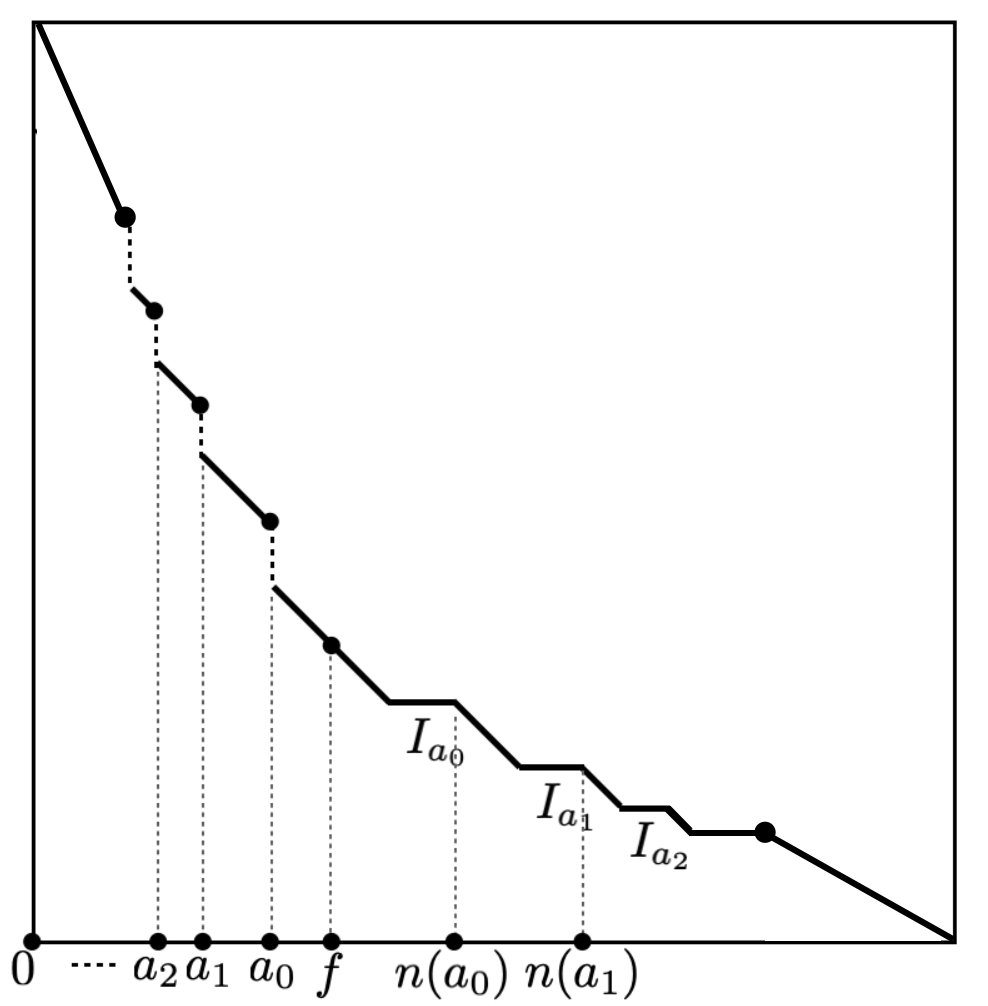}
        \caption{A negation function $n$ that determines the generic almost involutive $\mathsf{PQN}$-algebra ${\bf S}_n$. Notice that the set $E({\bf S}_n)$ of explosive elements of ${\bf S}_n$ is trivial and it coincides with $B({\bf S}_n)=\{0,1\}$.}
        \label{figNewNeg}
    \end{figure}


\begin{corollary}\label{corAIPQNsemilin}
    The varieties $\PQNbmaxi$ and $\PQNmaxbi$ are generated by their totally ordered members (i.e.~are semilinear). In fact, it holds that $\PQNbmaxi=\PQNmaxbi=\mathsf{ai}\PQNmaxe$.
\end{corollary}
\begin{proof} Let $\mathsf{aiV}$ stand for either $\PQNbmaxi$ or $\PQNmaxbi$. 
    By Lemma~\ref{lem:lemexplo}, for every $\langle {\bf A},\circ\rangle$ in $\mathsf{aiV}$ it holds that $E({\bf A})=B(\bf A)$ and therefore, by Lemma \ref{LemmaDireIndBmax}, $\langle {\bf A},\circ\rangle$ is directly indecomposable iff $E(A)=\{0,1\}$. Thus, by Lemma \ref{lemmaEA} if $\langle{\bf A},\circ\rangle$ is such that $E({\bf A})=\{0,1\}$,  ${\bf A}$ is totally ordered. Summing up, if $\langle {\bf A},\circ\rangle$ is directly indecomposable, then  ${\bf A}$ is totally ordered. Thus, $\mathsf{aiV}$ is generated by its totally ordered members.

Now, from what we observed above, there is only one way to define a max-consistency operator on chains, that also is a $\mathsf{maxB}$ and $\mathsf{Bmax}$. Thus, the varieties $\PQNbmaxi$ and $\PQNmaxbi$ are generated by the same chains and must therefore coincide.  Together with what we observed above, this observation gives us that $\PQNbmaxi=\PQNmaxbi$.
\end{proof}


\begin{remark}\label{remDelta} On  {\sf aiPQN}$^{\circ}_{\sf m}$-chains one can explicitly
define the Baaz 
operator $\Delta$
(given by $\Delta(1) = 1$ and $\Delta(x) = 0$ for $x \neq 1$; see \cite{Baaz96}).
Indeed, one can check that
$\Delta(a) := a \land \circ(a)$ works.  
For $a=1$, we have $\circ a=a=1$, hence $a\wedge\circ a=1$, while for $0 < a<1$ 
we have $\circ a=0$, so 
$a\wedge \circ a=0$; finally, for $a=0$, we obviously have $a\wedge \circ a=0$. 
Conversely, in the variety  {\sf aiPQN}$_{\Delta}$ of   {\sf aiPQN} algebras expanded with $\Delta$ (defined by the usual set of equations), the (only) {\sf max}-operator in a chain can be defined as $\circ(a) = \Delta(a \lor \neg a)$. Therefore, {\sf aiPQN}$^{\circ}_{\sf m}$-chains and  {\sf aiPQN}$_{\Delta}$ chains are term-equivalent, and thus  the variety generated by the class of {\sf aiPQN}$^{\circ}_{\sf m}$-chains coincides with the variety of {\sf aiPQN}$_{\Delta}$ algebras. Notice that for $\Delta$  to be definable on $\mathsf{PQN}$-chains, it is necessary that explosive and Boolean elements coincide. In fact, if $a$ not Boolean but $\circ a=1$, then $a\wedge \circ a=a\neq \Delta a=0$. In other words, condition (i) of Lemma \ref{lem:lemexplo} is key for the definition of $\Delta$ on $\PQNmax$ chains.
\end{remark}

Finally, adapting the argument 
of Theorem \ref{Sinfty}, we can obtain the following result.

\begin{theorem} \label{Sinfty2}
For every almost involutive $\mathsf{PQN}$-chain ${\bf S}_n$ that is generic for $\mathsf{ai}\PQNA$,  $\langle{\bf S}_n,\circ\rangle$ generates the variety $\PQNbmaxi=\PQNmaxbi$. 
\end{theorem}

We close the section with 
an observation about the preservation
of  semilinearity 
that is analogous
to the one made in the last part of the previous section 
in the case of $\mathsf{PQN}$-algebras with a undeterminedness operator.
Indeed, while none of the (quasi)varieties $\PQNmax_{\sf m}$, $\PQNmaxb$ and $\PQNbmax$  is semilinear, semilinearity is preserved on their sub-quasivarities satisfying  
the following quasi-equation:
\begin{equation}
\text{ if } (x \leftrightarrow y)\lor z = 1 , \text{ then } (\circ x \leftrightarrow \circ y)\lor z = 1
\tag{Con-$\vee^\circ$}
\end{equation}
For any $\mathsf{K} \in \{\PQNmax_{\sf m}, \PQNmaxb, \PQNbmax\}$, let us denote the corresponding
quasivariety by $\mathsf{K}\mbox{-}\lor$. 
For the same reasons explained in Section~\ref{recov-undeterminded}, each 
$\mathsf{K}\mbox{-}\lor$ is semilinear, and hence  generated by its totally ordered members. 
On the other hand, we have seen that the unique consistency operator that can be defined on a PQN chain is that of \eqref{maxconschain}. Hence, the chains in 
$\mathsf{K}\mbox{-}\lor$ for each $\mathsf{K}$ coincide;  thus 
we have the following result that can be interestingly compared with Proposition \ref{Prop:Var-V} above.

\begin{proposition}
$\PQNmax_{\sf m}\mbox{-}\lor = \PQNmaxb\mbox{-}\lor = \PQNbmax\mbox{-}\lor$. 
\end{proposition}

\section{$\PQNA$-algebras 
with both 
operators and almost involutive negations}
\label{sec:both}

We now briefly consider the case of $\mathsf{PQN}$-algebras with both recovery operators $\bu$ and $\circ$. First of all, let us observe that any algebra $\A$ satisfies the equation 
$\nnot\bu x \leq \circ x$.
Indeed,
for all $a \in A$, 
we have
the equality $ a \lor \nnot a \lor \bu a = 1$, 
which we can negate, 
obtaining 
$$
\nnot ( a \lor \nnot a \lor \bu a) = \nnot a \land \nnot \nnot a \land \nnot \bu a =  0 = \nnot  1.
$$
Since $ a \land  \nnot a \land \nnot \bu a \leq  \nnot a \land \nnot \nnot a \land \nnot \bu a $,
we obtain $ a \land  \nnot a \land \nnot \bu a =0$ which, by definition of $\circ$, 
gives us $\nnot \bu a \leq \circ a$.
The converse inequality, which need not hold in general,
is equivalent to the conditions in Lemma~\ref{lem:lemexplo}  (hence, it holds on $\mathsf{aiPQN}$-algebras), as the following results shows.
 
\begin{lemma} \label{lem:lemexplo2} 
For a quasi-Nelson algebra $\A$ endowed with both a $\max$-consistency operator and a 
$\min$-undeterminedness operator,
each of the following equations 
is equivalent to each of the conditions of Lemma~\ref{lem:lemexplo}:
\begin{enumerate}[(i)]
    \item[(iv)] 
    $\bu x \leq \bu (x \land \nnot x)$.
                \item[(v)] 
                $ \circ x = \nnot \bu x$.
                \item[(vi)] 
                $\circ x \leq \nnot \bu x $.
                \item[(vii)] 
                $\bu x \leq \nnot \circ x$.
\end{enumerate}
\end{lemma}
\begin{proof}
In this proof, when referring to (ii), we mean item (ii) of  Lemma~\ref{lem:lemexplo}.

     (ii)$\Leftrightarrow$(iv). Assuming (iv) holds, let $a \in A$ be such that $\nnot a = 0$. Then $\bu a \leq \bu (a \land \nnot a) = \bu 0 = 0$,
so $1 = a \lor \nnot a \lor \bu a = a \lor 0 \lor 0 = a$, as required. 
Conversely, assume Lemma \ref{lem:lemexplo} (ii) holds, and let $a \in A$.
Note that $a \land \nnot a \land \nnot \bu (a \land \nnot a) = 0$ always holds. 
Indeed, from
$1 = (a \land \nnot a ) \lor \nnot (a \land \nnot a ) \lor \bu (a \land \nnot a) = \nnot (a \land \nnot a ) \lor \bu (a \land \nnot a) $, we have
$0 
= \nnot 1 
= \nnot (\nnot (a \land \nnot a ) \lor \bu (a \land \nnot a)) 
= \nnot \nnot (a \land \nnot a ) \land \nnot \bu (a \land \nnot a)  
 = \nnot \nnot a \land \nnot \nnot \nnot a  \land \nnot \bu (a \land \nnot a)  
 =  a \land \nnot  a  \land \nnot \bu (a \land \nnot a)  $.
But $0 = a \land \nnot a \land \nnot \bu (a \land \nnot a) = \nnot (a \lor \nnot a \lor \bu (a \land \nnot a))$.
Thus, applying the quasi-equation, we obtain $a \lor \nnot a \lor \bu (a \land \nnot a) = 1$, which gives us
$\bu a \leq \bu (a \land \nnot a)$, as required. 


 (ii)$\Leftrightarrow$(v). Assuming Lemma \ref{lem:lemexplo} (ii), let $a \in A$. 
Since $a \land \nnot a = \nnot \nnot  a \land \nnot  a$
and $\nnot \nnot \circ a = \circ a$, we have
$0 
= \nnot a \land a \land \circ a 
= \nnot a \land \nnot \nnot a \land \nnot \nnot \circ a 
= \nnot (a \lor \nnot a \lor \nnot \circ a)$. Hence,
applying Lemma \ref{lem:lemexplo} (ii), 
we obtain 
$a \lor \nnot a \lor \nnot \circ a = 1$.
By the definition of $\bu$, this means that 
$\bu a \leq \nnot \circ a $, whence
$ \nnot  \nnot \circ a = \circ a \leq \nnot \bu a $.
As mentioned above, the converse inequality, $   \nnot \bu a \leq \circ a$, always holds.
Now assume (v) holds but Lemma \ref{lem:lemexplo} (ii) does not.
Then there is $a \in A$ such that
$a < 1$ and $\nnot a = 0$. Then $a \land \nnot a = 0$, so $\circ a = 1$.
By (v), we have $\nnot \bu a = \circ a = 1$, so $ \bu a \leq \nnot \nnot \bu a = 0 $.
By the definition of $\bu$, this in turn would imply $1 = a \lor \nnot a \lor \bu a = a \lor 0 = a $, against
our initial assumption.

The equivalence (v)$\Leftrightarrow$(vi) is clear.

(v)$\Leftrightarrow$(vii).
Recall that 
$x \leq \nnot \nnot x$ holds on every  commutative residuated lattice. 
Assuming 
$
\nnot \bu a = \circ a
$
 for some $a \in A$, we have 
$
\bu a \leq \nnot \nnot \bu a = \nnot \circ a
$.
Similarly, from $\bu a \leq  \nnot \circ a$
we obtain $ \circ a  \leq \nnot   \nnot \circ a \leq \nnot \bu a $,
and the inequality 
$   \nnot \bu a \leq \circ a  $ holds in general, as we have observed. 
\end{proof}

Item (v) in the above lemma is particularly interesting as it complements an observation we made in Section \ref{prelinearity} concerning involutivity. There, anticipating 
Corollary~\ref{corAIPQNsemilin}, we claimed that involutivity of the negation is a sufficient
but not necessary condition to recover the semilinearity of a variety of $\mathsf{PQN}$-algebras expanded with a $\mathsf{Bmax}$-consistency operator. Now, we have seen that the involutivity of the negation is sufficient but not necessary to define  $\circ$ from $\bu$. Indeed, by the equivalence of Lemma~\ref{lem:lemexplo} (ii) and Lemma~\ref{lem:lemexplo2} (v) it is the involutive behavior of $\nnot$ in a left interval of $1$ (recall Figure~\ref{figNewNeg}) that ensures that definability property.

The following proposition has a more limited application, namely to Boolean consistency operators.

\begin{proposition} \label{lem:bmaxiitoiii} Let $\Al$ be a  quasi-Nelson algebra 
endowed with  a $\rm{B} \!  \max $-consistency operator $\circ$
and a 
$\min$-undeterminedness operator $\bu$
that
satisfies 
any of the equations in
Lemma~\ref{lem:lemexplo2}. 
Then $\bu$ is a $\mathsf{Bmin}$ 
operator
and, in consequence, 
$\Al$ also  satisfies $\bu x = \nnot \nnot \bu x$ 
and 
$\bu x = \nnot \circ x$. 
\end{proposition} 

\begin{proof}
%
Let $a \in A$ be a generic element. Let us show that $\bu a \in B(\Al)$,
which entails $\bu a = \nnot \nnot \bu a$. 
Indeed, from $\circ a \lor \nnot \circ a = 1$, we have
$\nnot (\circ a \lor \nnot \circ a ) = \nnot \circ a \land \nnot \nnot \circ a  = \nnot \circ a \land  \circ a = 0$. Using the assumption $ \bu a \leq \nnot \circ a$, we thus obtain
$\bu a \land  \circ a = 0$. Moreover from $\nnot \bu a \leq \circ a$
and $\nnot \circ a \land  \circ a = 0$
we obtain $0 = \nnot \circ a \land  \nnot \bu a = \nnot (\circ a \lor \bu a)$. 
From the latter, applying the quasi-equation in Lemma~\ref{lem:lemexplo} (ii),
we obtain $\circ a \lor \bu a = 1$. Hence, $\circ a$ and $\bu a$ are Boolean complements
of each other. Thus, $\bu a \in B(\Al)$ and $\bu a = \nnot \nnot \bu a$. 
Then $\nnot \bu a \leq \circ a$  also yields
$\nnot \circ a \leq \nnot \nnot \bu a = \bu a $, as required.
\end{proof} 

The condition $\bu x = \nnot \nnot \bu x$ considered in the conclusion of the preceding proposition 
has some independent interest:
we may observe that, in general, it does not imply  
the conditions of Lemma~\ref{lem:lemexplo2}.
Indeed, 
considering the three-element G\"odel chain $\Al[G]_3$ (with elements $0 < a < 1$) endowed with the corresponding operators
(cf.~Remark~\ref{rem:pseudocowithball}), we see that $\Al[G]_3$ satisfies 
$\bu x = \nnot \nnot \bu x$.
However, 
$1 = \bu a  \not \leq \nnot \circ a = 0$.

\section{
The logic $\mathcal{PQN}$
with an undeterminedness operator}\label{sec:logUndeter}

In this and the following sections we move from the algebraic framework we have considered so far to a logical setting to analyze the logics resulting from expanding the prelinear quasi-Nelson logic $\mathcal{PQN}$ with recovery operators. 

In this section we start with expanding the logic $\mathcal{PQN}$ with an undetermindedness operator $\bu$,  
%
%
with the idea in mind of eventually capturing the logics behind the varieties $\mathsf{PQN}^\bullet_{\sf m}, \mathsf{PQN}^\bullet_{\sf Bm}$ and $\mathsf{PQN}^\bullet_{\sf mB}$ of prelinear quasi-Nelson algebras endowed with a $\min$, $\rm{B}\!\min $ and $\min\!\rm{B}$ operators respectively, as studied in Section \ref{recov-undeterminded}.

\begin{definition} \label{def51} 
(i) The logic  $\mathcal{PQN}^\bullet_{\sf m}$ is defined as the expansion of $\mathcal{PQN}$ in a language with a new unary connective $\bullet$ with  the following additional axiom: \vspace{0.2cm}

\begin{tabular}{ll}
{\rm(A1$_\bu$)} & $\varphi\lor \neg\varphi \lor \bullet\varphi$ \\
\end{tabular}
\vspace{0.2cm}

\noindent
and inference rules:  \vspace{0.2cm}

\begin{tabular}{ll}
{\rm(Cong$_\bu$)} \, $\displaystyle \frac{ \varphi \leftrightarrow \psi}{ \bullet\varphi \to \bullet\psi}$ \hspace{1cm} 
& {\rm(Min)} \,$\displaystyle \frac{\varphi\lor \neg\varphi \lor \psi}{ \bullet\varphi \to \psi}$.   \\
\end{tabular}
\vspace{0.2cm}

\noindent (ii) The logic  $\mathcal{PQN}^\bullet_{\sf Bm}$ is defined as the axiomatic extension of $\mathcal{PQN}^\bullet_{\sf m}$ with the additional axiom 

\begin{tabular}{ll}
{\rm(A2$_\bu$)}  & ${\bullet} \varphi \lor \neg {\bullet} \varphi $ \\
\end{tabular}

\noindent (iii) Finally, the logic  $\mathcal{PQN}^\bullet_{\sf mB}$ is defined as the expansion of $\mathcal{PQN}$ with the axioms (A1$_\bu$), (A2$_\bu$), the rule (Cong$_\bu$) and the following rule: 

\vspace{0.2cm}

\begin{tabular}{ll}
 {\rm(MinB)} \,$\displaystyle \frac{\varphi\lor \neg\varphi \lor \psi, \quad \psi \lor \neg \psi}{ \bullet\varphi \to \psi}.$  \vspace{0.2cm} \\
\end{tabular}
\end{definition}

Note that it is the presence of the (Cong$_\bu$) inference rule in the above three logics that guarantees the algebraizability  of these logics (in the sense of Blok-Pigozzi).
Moreover, thanks to \cite[Thm.~1 and Thm.~3]{FiRi24}, 
the inference rules (Min) and (MinB) 
can be safely replaced by the following two axioms:

\begin{enumerate}[(A3$_\bu$)]
\item $\bu \varphi  \to (\psi \lor \bu  (\varphi \lor \nnot \varphi \lor \psi))$
\item[(A4$_\bu$)] $\nnot \bu \top$, 
 \end{enumerate}
while the (MinB) inference rule can be safely replaced by the axiom (A4$_\bu$) plus the following axiom:
\begin{enumerate}
\item[(A5$_\bu$)]   $\bu \varphi  \to (\psi \lor \bu (\psi \lor \nnot \psi) \lor \bu  (\varphi \lor \nnot \varphi \lor \psi))$.
 \end{enumerate}
 
Therefore, the logics $\mathcal{PQN}^\bullet_{\sf m}$, $\mathcal{PQN}^\bullet_{\sf mB}$ and $\mathcal{PQN}^\bullet_{\sf Bm}$ are automatically sound and strongly complete with respect to their corresponding varieties of algebras $\mathsf{PQN}^\bullet_{\sf m}, \mathsf{PQN}^\bullet_{\sf Bm}$ and $\mathsf{PQN}^\bullet_{\sf mB}$ of prelinear quasi-Nelson algebras with an undeterminedness operator, respectively. 

We note that the logic $\mathcal{PQN}$ from which we departed is paracomplete,
i.e.~does not prove the Excluded Middle principle $
\varphi \lor \neg \varphi$. To see this it is enough, by completeness, to find a $\mathsf{PQN}$-chain where the equation $x \lor \neg x = 1$ is not valid: for instance, this is obviously the case of the (generic) $\mathsf{PQN}$-algebra ${\bf S}_n$ (Figure \ref{figNewNeg}). Since on any $\mathsf{PQN}$-chain it is always possible to define a $\mathsf{max}$-underterminedness operator $\bullet$ that is also a $\mathsf{maxB}$ and $\mathsf{Bmax}$ operator as in~\eqref{eqBMinChain}, the same (counter-)example shows that $\mathcal{PQN}^\bullet_{\sf m}$, $\mathcal{PQN}^\bullet_{\sf mB}$ and $\mathcal{PQN}^\bullet_{\sf Bm}$ are also paracomplete. 

\begin{proposition}
    $\mathcal{PQN}^\bullet_{\sf m}$, $\mathcal{PQN}^\bullet_{\sf mB}$ and $\mathcal{PQN}^\bullet_{\sf Bm}$ are Logics of Formal Underterminedness (LFU). 
\end{proposition}






Theorem \ref{thm:PQNminSemiLin} tells us that the variety $\mathsf{PQN}^{\bu}_{\mathsf{Bm}}$ is semilinear while $\mathsf{PQN}^{\bu}_{\mathsf{m}}$ and $\mathsf{PQN}^{\bu}_{\mathsf{mB}}$ are not. 
However, by results of Cintula and Noguera~\cite[Corollary 2.6.5]{Cintula-Noguera-book}, it follows that any expansion of a semilinear logic where the new inference rules are closed by disjunction keeps being semilinear. Therefore, since $\mathsf{PQN}^{\bu}_{\mathsf{m}}$ and $\mathsf{PQN}^{\bu}_{\mathsf{mB}}$ are expansions of the variety $\mathsf{PQN}$, which is semilinear, in order for these logics to be complete with respect to the class of totally ordered members of its corresponding variety or quasi-variety, it is enough to replace its inference rules by their $\lor$-closed form.

Since the only inference rule involved in the axiomatic systems of these two logics (and that in $\mathsf{PQN}^{\bu}_{\mathsf{Bm}}$ as well) is the rule  (Cong$_\bu$),\footnote{Recall that, as mentioned in the previous subsection, the rules (Min) and (MinB) 
appearing in the definitions of the logics $\mathcal{QN}^\bullet_{\sf m}$ and in $\mathcal{QN}^\bullet_{\sf mB}$, respectively, can be safely replaced by axioms.} we can consider their semilinear closure by replacing this rule by its $\lor$-closed form: 
\begin{equation}
    \tag{Cong$_\bu^\lor$} \frac{ (\varphi \leftrightarrow \psi) \lor \chi}{(\bu\varphi \to \bu\psi) \lor \chi}
\end{equation}


Then, by results of Cintula and Noguera \cite[Corollary 2.6.5]{Cintula-Noguera-book}, it follows that the resulting logics are semilinear, and hence complete with respect to their corresponding classes of chains. 
Namely, if we let 
$\mathcal{PQN}^{\bu\lor}_{\mathsf{m}}$, $\mathcal{PQN}^{\bu\lor}_{\mathsf{Bm}}$ and  $\mathcal{PQN}^{\bu\lor}_{\mathsf{mB}}$ denote the logics with the $\lor$-closed form of the (Cong) rule, then we have the following chain-completeness results.

\begin{theorem}  \label{chaincompletebu}
For any $\Lc  \in \{ 
\mathcal{PQN}^{\bu\lor}_{\mathsf{m}}, \mathcal{PQN}^{\bu\lor}_{\mathsf{Bm}} = \mathcal{PQN}^{\bu}_{\mathsf{Bm}}, \mathcal{PQN}^{\bu\lor}_{\mathsf{mB}}\}$, $\Lc$  is sound and complete with respect to the corresponding class of  totally ordered $L$-algebras. 
\end{theorem}

Note that, since the variety $\mathsf{PQN}^{\bu}_{\mathsf{Bm}}$ is semilinear, it is coincides with the equivalent algebraic semantics of the logic $\mathcal{PQN}^{\bu\lor}_{\mathsf{Bm}}$, hence $\mathcal{PQN}^{\bu\lor}_{\mathsf{Bm}} = \mathcal{PQN}^{\bu}_{\mathsf{Bm}}$. 

On the other hand, as  mentioned in Section~\ref{prelinearity}, there is only one way to define a $\mathsf{min}$-, $\mathsf{Bmin}$- and $\mathsf{minB}$-undeterminedness operator on a $\mathsf{PQN}$-chain $\bf A$, namely by letting $\hat\bu a = 0$ for $a \in B({\bf A})$ and $\hat\bu a = 1$ otherwise. Moreover, it is easy to check that the well-known Baaz-Monteiro projection operator $\Delta$ is inter-definable with the operator $\hat\bu$, indeed $\Delta x = x \land \nnot \hat\bu x$ and conversely, $\hat\bu x = \nnot\Delta(x \lor \nnot x)$. This is another difference with the consistency $\hat\circ$ operator. These considerations
entail  that the class of chains in the varieties $\mathsf{PQN}^{\bu}_{\mathsf{m}}$, $\mathsf{PQN}^{\bu}_{\mathsf{Bm}}$ and $\mathsf{PQN}^{\bu}_{\mathsf{mB}}$ is the same. Hence, due to the above completeness result, it turns out that the three corresponding semilinear logics collapse. 

\begin{corollary}\label{corPQN64}
 $\mathcal{PQN}^{\bu}_{\mathsf{Bm}} = \mathcal{PQN}^{\bu\lor}_{\mathsf{Bm}} = \mathcal{PQN}^{\bu\lor}_{\mathsf{m}} = \mathcal{PQN}^{\bu\lor}_{\mathsf{mB}}$. 
\end{corollary}
Observe that if  $\varphi$ a formula in the language of $\mathcal{PQN}$  that is not valid in a $\mathsf{PQN}$-chain ${\bf A}$, then $\varphi$ will not be valid  in  $\langle{\bf A},\hat\bu\rangle$ either. Hence, by Corollary \ref{corPQN64}, it follows that all the logics $\mathcal{PQN}^{\bu}_{\mathsf{Bm}}, \mathcal{PQN}^{\bu\lor}_{\mathsf{Bm}}, \mathcal{PQN}^{\bu\lor}_{\mathsf{m}}$ and $\mathcal{PQN}^{\bu\lor}_{\mathsf{mB}}$ are conservative expansions of $\mathcal{PQN}$. 


Moreover, by Theorem \ref{Sinfty}, these logics are in fact complete with respect to a single chain.

\begin{corollary}
 $\mathcal{PQN}^{\bu}_{\mathsf{Bm}}$
 is complete with respect to any single chain $\langle {\bf A}, \hat\bu\rangle$, where ${\bf A} \in \mathbb{ S}_\infty$. 
\end{corollary}

Finally, 
it is worth remarking that all the logics considered in this section, namely
$\mathcal{PQN}^{\bu},  \mathcal{PQN}^{\bu}_{\mathsf{m}}, \mathcal{PQN}^{\bu}_{\mathsf{Bm}}, \mathcal{PQN}^{\bu}_{\mathsf{mB}}$, $\mathcal{PQN}^{\bu\lor},  \mathcal{PQN}^{\bu\lor}_{\mathsf{m}}, \mathcal{PQN}^{\bu\lor}_{\mathsf{Bm}}$, and $\mathcal{PQN}^{\bu\lor}_{\mathsf{mB}}$, which are all truth-preserving and paracomplete, are logics of formal undeterminedness (LFUs) 
expanding $\mathcal{PQN}$.

\section{
The logic $\mathcal{PQN}$ with a consistency operator}\label{sec:logCons}


In this section we now consider the expansions of 
the prelinear quasi-Nelson logic $\mathcal{PQN}$ with  a consistency operator and their paraconsistent degree-preserving and non-falsisty preserving companions. 


\begin{definition} \label{def63} 
(i) The logic  $\mathcal{PQN}^\circ_{\sf m}$ is defined as the expansion of $\mathcal{PQN}$ in a language with a new unary connective $\circ$ with  the following additional axiom: \vspace{0.2cm}

\begin{tabular}{ll}
{\rm(A1$_\circ$)} & $\nnot(\varphi \land {\nnot} \varphi \land \circ \varphi)$ \\
\end{tabular}
\vspace{0.2cm}

\noindent
and inference rules:  \vspace{0.2cm}

\begin{tabular}{ll}
{\rm(Cong$_\circ$)} \, $\displaystyle \frac{ \varphi \leftrightarrow \psi}{ \circ\varphi \to \circ\psi}$ \hspace{1cm} 
& {\rm(Max)} \,$\displaystyle \frac{\nnot(\varphi \land {\nnot} \varphi \land \psi)}{ \psi \to \circ\varphi}$.  \vspace{0.2cm} \\
\end{tabular}

\noindent (ii) The logic  $\mathcal{PQN}^\circ_{\sf Bm}$ is defined as the axiomatic extension of $\mathcal{PQN}^\circ_{\sf m}$ with the additional axiom 

\begin{tabular}{ll}
{\rm(A2$_\circ$)}  & ${\circ} \varphi \lor \nnot {\circ} \varphi $ \\
\end{tabular}

\noindent (iii) The logic  $\mathcal{PQN}^\circ_{\sf mB}$ is defined as the expansion of $\mathcal{PQN}$ with the axioms (A1$_\circ$), (A2$_\circ$), the rule (Cong$_\circ$) and the following rule: 

\vspace{0.2cm}

\begin{tabular}{ll}
 {\rm(MaxB)} \,$\displaystyle \frac{\nnot(\varphi \land \nnot \varphi \land \psi), \quad \psi \lor \nnot \psi}{ \psi \to \circ\varphi}.$  \vspace{0.2cm} \\
\end{tabular}


\end{definition}


Again, since they keep being algebraizable due to the rule (Cong$_\circ$), it directly follows from the fact that rules (Max) and (MaxB) can be equivalently replaced by axioms\footnote{Following the algebraic results in Section \ref{prelinearity}.}
that these logics are sound and complete with respect to their equivalent algebraic semantics, given respectively by the varieties $\mathsf{PQN}^{\circ}_{\mathsf{m}}$, $\mathsf{PQN}^{\circ}_{\mathsf{Bm}}$ 
and 
$\mathsf{PQN}^{\circ}_{\mathsf{mB}}$.

We have seen in Section \ref{prelinearity}  that none of the varieties $\mathsf{PQN}^{\circ}_{\mathsf{m}}$, $\mathsf{PQN}^{\circ}_{\mathsf{Bm}}$ and $\mathsf{PQN}^{\circ}_{\mathsf{mB}}$ 
are semilinear. Anyway, all of them are expansions of the variety $\mathsf{PQN}$, which is semilinear. Then, again by results of Cintula and Noguera~\cite[Corollary 2.6.5]{Cintula-Noguera-book}, 
in order for any of the logics above to be complete with respect to the class of totally ordered members of its corresponding variety or quasi-variety it is enough to replace its inference rules by their $\lor$-closed form. More in detail, let us consider the following $\lor$-closed  inference rules:

(Cong$_\circ^\lor$) $\displaystyle{\frac{ (\varphi \leftrightarrow \psi) \lor \chi}{(\circ\varphi \to \circ\psi) \lor \chi}}$

(Max$^\lor$) $\displaystyle{\frac{ \nnot (\varphi \land \nnot \varphi \land \psi) \lor \chi }{(\psi \to \circ \varphi) \lor \chi}}$

(MaxB$^\lor$): $\displaystyle{\frac{ (\nnot (\varphi \land \nnot \varphi \land \psi) \land (\psi \lor \nnot \psi)) \lor \chi}{(\psi \to \circ \varphi) \lor \chi}}$


\noindent Then we define the following semilinear companions of the above logics: 

\begin{itemize}
\item The logic $\mathcal{PQN}^{\circ\lor}_{\mathsf{m}}$ is obtained from $\mathcal{PQN}^{\circ}_{\mathsf{m}}$ by replacing the rule  (Cong$_\circ$) by the rule (Cong$_\circ^\lor$) and the rule (Max) by the rule (Max$^\lor$). 
\item The logic $\mathcal{PQN}^{\circ\lor}_{\mathsf{Bm}}$ is obtained as the axiomatic extension of $\mathcal{PQN}^{\circ\lor}_{\mathsf{m}}$ by the axiom {\rm(A2$_\circ$)}.

\item The logic  $\mathcal{PQN}^{\circ\lor}_{\mathsf{mB}}$ is obtained from $\mathcal{PQN}^{\circ}_{\mathsf{Bm}}$ by replacing and the rules (Cong$_\circ$) and (MaxB) by the rules (Cong$^\lor$) and (MaxB$^\lor$) respectively.
 

\end{itemize}

Then we have the following chain-completeness results.

\begin{theorem}  \label{chaincomplete}
For any $\Lc \in \{  \mathcal{PQN}^{\circ\lor}_{\mathsf{m}}, \mathcal{PQN}^{\circ\lor}_{\mathsf{Bm}}, \mathcal{PQN}^{\circ\lor}_{\mathsf{mB}} 
\}$, $\Lc$  is sound and complete with respect to the class of totally ordered $\Lf$-algebras. 
\end{theorem}

As  mentioned in Section \ref{prelinearity}, there is only one way to define a $\mathsf{max}$-, $\mathsf{Bmax}$- and $\mathsf{maxB}$-consistency
operator on a $\mathsf{PQN}$-chain $\bf A$, namely the operator $\hat\circ$ defined as $\hat\circ a = 1$ for $a \in \{0\} \cup E(A)$ and $\hat\circ a = 0$ otherwise. Then it follows that the class of chains in the varieties $\mathsf{PQN}^{\circ}_{\mathsf{m}}$ and $\mathsf{PQN}^{\circ}_{\mathsf{Bm}}$ 
and in the quasivariety $\mathsf{PQN}^{\circ}_{\mathsf{mB}}$ is the same. Hence, due to the above completeness result, it turns out that the four corresponding semilinear logics coincide,

\begin{corollary} \label{same}
 $\mathcal{PQN}^{\circ\lor}_{\mathsf{m}}$ = $\mathcal{PQN}^{\circ\lor}_{\mathsf{Bm}}$ = $\mathcal{PQN}^{\circ\lor}_{\mathsf{mB}}$ 
\end{corollary}
 The same argument used after Corollary \ref{corPQN64} shows that    the logics $\mathcal{PQN}^{\circ\lor}_{\mathsf{m}}, \mathcal{PQN}^{\circ\lor}_{\mathsf{Bm}}$ and $\mathcal{PQN}^{\circ\lor}_{\mathsf{mB}}$ are also conservative expansions of $\mathcal{PQN}$.

By Corollary \ref{same}, from now on, we will only refer to $\mathcal{PQN}^{\circ\lor}_{\mathsf{m}}$.  
Moreover, following a parallel reasoning of Theorem \ref{Sinfty}, these logics are in fact complete with respect to a single chain.

\begin{corollary}
 $\mathcal{PQN}^{\circ\lor}_{\mathsf{m}}$
 is complete with respect to any single chain $\langle {\bf A}, \hat\circ\rangle$, where ${\bf A} \in {\bf S}_\infty$. 
\end{corollary}

\subsection{Degree-preserving companion of 
$\mathcal{PQN}$ with a {\sf max}-consistency operator
}\label{sec:logCons1}

Going back to the main question of defining LFIs with semantics over $\mathsf{PQN}$-chains, note that the logics in the above Theorem \ref{chaincomplete} are not LFIs since they are all truth-preserving and hence explosive.  In order to get LFIs we shall consider two paraconsistent variants of them, namely the degree-preserving companions \cite{bou2009logics} in this subsection and the non-falsity preserving  companions \cite{GEGC} of them in the next subsection. 

\begin{definition} 
The degree-preserving companion of $\mathcal{PQN}^{\circ\lor}_{\mathsf{m}}$, denoted  $(\mathcal{PQN}^{\circ\lor}_{\mathsf{m}})^\leq$, is defined by the following axioms and rules: 

\begin{itemize}
\item[-] Axioms of $\mathcal{PQN}^{\circ\lor}_{\mathsf{m}}$
\item[-] Adjunction rule: (Adj) \quad $\displaystyle{\frac{\varphi, \quad \psi}{\varphi \land \psi}} $
\item[-] Restricted modus ponens rule: (r-MP) \quad $\displaystyle{\frac{\varphi, \quad  \mathcal{PQN}^{\circ\lor}_{\mathsf{m}} \vdash  \varphi \to \psi}{\psi}} $
\item[-] Restricted congruence rule: (r-Cong$^\lor$)\quad $\displaystyle{\frac{ \mathcal{PQN}^{\circ\lor}_{\mathsf{m}} \vdash (\varphi \leftrightarrow \psi) \lor \chi}{(\circ\varphi \to \circ\psi) \lor \chi}}$
\end{itemize}
\end{definition}

\begin{theorem}[Completeness]
$(\mathcal{PQN}^{\circ\lor}_{\mathsf{m}})^\leq$ is sound and complete with respect to the intended linear semantics, that is, for any finite set of formulas $\Gamma \cup \{\varphi\}$, $\Gamma \vdash_{(\mathcal{PQN}^{\circ\lor}_{\mathsf{m}})^\leq} \varphi$ iff, for every $\mathsf{PQN}^{\circ}_{\mathsf{m}}$-chain $\bf A$ and $\bf A$-evaluation $e$, $\min\{e(\psi) \mid \psi \in \Gamma\} \leq e(\varphi)$. 
\end{theorem}

The proof directly follows from Proposition \ref{axlessequal}. 
Therefore, since $\varphi,  {\sim} \varphi \not\vdash_{(\mathcal{PQN}^{\circ\lor}_{\mathsf{m}})^\leq} \bot$, we finally get that the degree-preserving logic $(\mathcal{PQN}^{\circ\lor}_{\mathsf{m}})^\leq$ is a LFI.
\begin{corollary}  $(\mathcal{PQN}^{\circ\lor}_{\mathsf{m}})^\leq$ is a Logic of Formal Inconsistency. 
\end{corollary}
Recalling  conditions~\eqref{eq:fgpe2} and~\eqref{eq:fgpe3}  from the Introduction, the  definition of LFIs (e.g.~\cite[Def.~3.1]{dIRL}) requires
that there exists a formula $\varphi$ such that 
$\varphi, \circ \varphi \not \vdash \bot$ and 
$\nnot \varphi, \circ \varphi \not \vdash \bot$. Such conditions are clearly preserved by each weakening
of a given logic: thus $(\mathcal{PQN}^{\circ\lor}_{\mathsf{m}})^\leq$ inherits them from its involutive extension considered in~\cite[Proposition 5.6]{dIRL}.

\subsection{Non-falsity preserving companion of $\mathcal{PQN}$ with a $\sf max$-consistency operator}\label{sec:logCons2}

We have seen  that all the logics $\mathcal{PQN}^{\circ\lor}_{\mathsf{m}}, \mathcal{PQN}^{\circ\lor}_{\mathsf{Bm}}$ and $\mathcal{PQN}^{\circ\lor}_{\mathsf{mB}}$ 
coincide and are semilinear (Corollary \ref{same}).
In the following we thus restrict ourselves to just one of them, say $\mathcal{PQN}^{\circ\lor}_{\mathsf{m}}$. 

Here we are interested in the paraconsistent variant of $\Lc = \mathcal{PQN}^{\circ\lor}_{\mathsf{m}}$ that preserves the non-falsity. 

\begin{definition} Let $\Lc = \mathcal{PQN}^{\circ\lor}_{\mathsf{m}}$  and let $\Gamma \cup \{\varphi\}$ be a finite set of formulas. We define: 


- $\Gamma \models^{(0}_{\Lc} \varphi$ whenever, for any $\Lf$-chain {\bf A} and any $\bf A$-evaluation $e$, if $e(\psi) > 0$ for every $\psi \in \Gamma$, then $e(\varphi) > 0$.

\end{definition}

The crucial observation here is the following lemma that shows how to express valid inferences in the logic $\models^{(0}_{\Lc}$ in terms of inferences in the 1-preserving logic $\models_{\Lc}$.

\begin{lemma} \label{lemma0} For $\Lc = \mathcal{PQN}^{\circ\lor}_{\mathsf{m}}$, the following conditions are equivalent:
\begin{center}
$\varphi \models^{(0}_{\Lc} \psi $
\quad iff \quad
 $\neg \psi \models_{\Lc} \neg \varphi$
\quad iff \quad
 $\models_{\Lc} (\neg \psi)^2 \to \neg \varphi$.
 \end{center}
\end{lemma}

\begin{proof} $\varphi \models^{(0}_{\Lc} \psi $ iff for any evaluation $e$, $e(\varphi) > 0$ implies $e(\psi) > 0$. Since in $\mathsf{PQN}$-chains it holds that $x > 0$ iff $\neg x < 1$, we have $\varphi \models^{(0}_{\Lc} \psi$ iff for any evaluation $e$, $e(\neg \varphi) < 1$ implies $e(\neg \psi) < 1$, that is, $e(\neg \psi) =1$ implies $e(\neg \varphi) = 1$, that is, iff  $\neg \psi \models_{\Lc} \neg \varphi$. Moreover, since the logic $\mathcal{PQN}$ is 3-potent (i.e.\ $\mathcal{PQN}$ proves  $\varphi * \varphi \to \varphi * (\varphi * \varphi)$), it has the following form of global deduction theorem: $\neg \psi \models_{\Lc} \neg \varphi$ iff $\models_{\Lc} (\neg \psi)^2 \to \neg \varphi$. 
\end{proof}

\begin{definition} The non-falsity variant of the logic logic $\Lc = \mathcal{PQN}^{\circ\lor}_{\mathsf{m}}$, denoted {\sf nf}-$\mathcal{PQN}^{\circ\lor}_{\mathsf{m}}$, is defined by the following set of axioms and rules:
\begin{itemize}
\item Axioms of  $\mathcal{PQN}^{\circ\lor}_{\mathsf{m}}$\vspace{0.1cm}
\item Restricted forms of the rules of  $\mathcal{PQN}^{\circ\lor}_{\mathsf{m}}$: \\ 

\begin{center}

\mbox{} (r-Cong$_\circ^\lor$) $\displaystyle{\frac{\vdash_{\Lc} (\varphi \leftrightarrow \psi) \lor \chi}{(\circ\varphi \to \circ\psi) \lor \chi}}$
\quad\quad\quad
(r-Max$^\lor$) $\displaystyle{\frac{ \vdash_{\Lc} \nnot (\varphi \land \nnot \varphi \land \psi) \lor \chi }{(\psi \to \circ \varphi) \lor \chi}}$ \mbox{}
\end{center}
\item Adjunction: $\displaystyle{\frac{\varphi, \quad \psi}{\varphi \land \psi}} $
\item Restricted nf-MP: $\displaystyle{\frac{\varphi, \quad  \vdash_{\Lc}  (\neg \psi)^2 \to \neg \varphi}{\psi}} $
\end{itemize}
\end{definition}

\begin{remark} The rule of Restricted Modus Ponens r-MP:  $\displaystyle{\frac{\varphi, \quad  \vdash_{\Lc}  \varphi \to \psi}{\psi}} $
 is derivable in the logic {\sf nf}-PQN, indeed from the rule Restricted nf-MP and adjunction. It follows from the observation that $\varphi \to \psi \vdash_{\Lc} (\neg \psi)^2 \to \neg \varphi$. 
\end{remark}

\begin{theorem}   For any finite set of formulas $\Gamma \cup \{\varphi\}$, $\Gamma \models^{(0}_{\Lc} \varphi$  iff $\Gamma \vdash_{{\sf nf}-\mathcal{PQN}^{\circ\lor}_{\mathsf{m}}} \varphi$. 
\end{theorem}

\begin{proof} Soundness is easy. Suppose $\Gamma \models^{(0}_{\Lc} \varphi$, and let $\psi = \bigwedge\{ \chi \mid \chi \in \Gamma\}$. By Lemma \ref{lemma0}, $\psi \models_{\Lc}^{(0} \varphi$ iff $\models_{\Lc} (\neg \varphi)^2 \to \neg \psi$. By completeness of $\mathcal{PQN}^{\circ\lor}_{\mathsf{m}}$, 
we have  $\vdash_{\mathcal{PQN}^{\circ\lor}_{\mathsf{m}}}  (\neg \varphi)^2 \to \neg \psi$. This means that there is a proof
$$ \langle \Psi_1, \ldots, \Psi_n =  (\neg \varphi)^2 \to \neg \psi \rangle$$
where each $\Psi_i$ is an axiom or has been derived from previous elements of the proof by an application of the inference rules. In order to get a proof of $\varphi$ from $\Gamma$  in {\sf nf}-$\mathcal{PQN}^{\circ\lor}_{\mathsf{m}}$, it roughly suffices to  consider two additional steps: 

$\Psi_0 = \psi$, by adjunction from $\Gamma$

$\Psi_{n+1} =  \varphi$, by application of the  Restricted nf-MP inference rule to $\Psi_0$ and $\Psi_n$. 

\noindent Note that the inference rules used in the proof in $\mathcal{PQN}^{\circ\lor}_{\mathsf{m}}$ apply always to theorems, so there can also be applied (as restricted rules) in {\sf nf}-$\mathcal{PQN}^{\circ\lor}_{\mathsf{m}}$. Therefore,  $\langle \Psi_0, \Psi_1, \ldots, \Psi_n, \Psi_{n+1} =  \varphi \rangle$ is indeed a proof of $\varphi$ from $\Gamma$ in {\sf nf}-$\mathcal{PQN}^{\circ\lor}_{\mathsf{m}}$, and hence  $\Gamma \vdash_{{\sf nf\mbox{-}}\mathcal{PQN}^{\circ\lor}_{\mathsf{m}}} \varphi$.
\end{proof}

\subsection{Propagation properties and recovering classical logic}\label{sec:propagation}

One of the remarkable features in some logics of formal inconsistency, as in da Costa's C-systems, is the so-called propagation property of the consistency connective $\circ$. Namely, the fact that consistency, or determinedness or well-behaviour, is propagated by means of the connectives, in the following sense: (i) if  $\varphi$ is consistent  it follows that $\nnot \varphi$ is consistent too, and (ii) if besides $\psi$ is also consistent then it follows $\varphi \# \psi$ is consistent as well for every binary connective $\#$. In this final subsection we examine the status of these propagation properties both for the duals of the undeterminedness operators $\bu^\delta$ and for the consistency operators $\circ$ we have considered in this work. 
We will make a semantical/algebraic analysis in the case of the semilinear LFU and LFI  logics we have defined in this section.  

We start with the LFU logic  $\mathcal{PQN}^{\bu}_{\mathsf{Bm}}$  and operators $\bu^\delta = \nnot \bu$, where $\bu$ is a {\sf min}-undertermindedness operator. We have seen that in a $\mathsf{PQN}$-chain $\bf A$ there is a unique such operator, defined as 
$$ \bu(x) = \left \{
\begin{array}{ll}
0, & \mbox{if } x \in \{0, 1\}  \\
1, & \mbox{otherwise, }
\end{array}
\right . $$
and therefore its corresponding dual operator is defined as
$$ \bu^\delta(x) = \left \{
\begin{array}{ll}
1, & \mbox{if } x \in \{0, 1\}  \\
0, & \mbox{otherwise. }
\end{array}
\right . $$
Hence, in the logic  $\mathcal{PQN}^{\bu}_{\mathsf{Bm}}$, a formula like $\bu^\delta \varphi$ is 1-true iff $\varphi$ is either fully true or fully false, i.e. it behaves as a Boolean formula. 
Now, since the truth-functions of the logic  $\mathcal{PQN}^{\bu}_{\mathsf{Bm}}$ over the top and bottom elements of any totally ordered algebra  coincide with the classical 2-valued Boolean truth-functions,  the following first propagation properties hold.  

\begin{proposition} The determinedness operator $\bu^\delta$ in the paracomplete logic $\mathcal{PQN}^{\bu}_{\mathsf{Bm}}$ satisfies the propagation property with respect to the rest of the connectives, that is, for any formulas $\varphi, \varphi_1, \varphi_2$, the following conditions hold: 
\begin{itemize}
\item[-] $ \bu^\delta \varphi \vdash_{\mathcal{PQN}^{\bu}_{\mathsf{Bm}}} \bu^\delta {\nnot}\varphi $
\item[-] $ \bu^\delta \varphi_1,  \bu^\delta \varphi_2  \vdash_{\mathcal{PQN}^{\bu}_{\mathsf{Bm}}} \bu^\delta (\varphi_1 \# \varphi_2),  \quad \mbox{ for } \# \in \{\land, \lor, \ast, \to\}$
\end{itemize}
\end{proposition}
In fact, this result can be seen as a particular case of \cite[Prop. 5.2]{CEG14} since the
operator $\bu^{\delta}$ coincides with the consistency operator $\circ_{\min}$ considered in that reference in the context of LFIs defined on top of fuzzy logics endowed with a consistency operator.\footnote{Notice that, although the operator $\bu^\delta$ behaves like a consistency operator and is definable in the logic $\mathcal{PQN}^{\bu}_{\mathsf{Bm}}$, this logic is not an LFI for $\bu^\delta$ since, in particular, it is explosive, and thus, not paraconsistent.} 


In the case of the semilinear logic of formal inconsistency $(\mathcal{PQN}^{\circ\lor}_{\mathsf{m}})^\leq$, since it  is a degree-preserving logic, the conditions to be checked are whether the formulas
\begin{itemize}
\item[(i)] $ \circ \varphi \to \circ {\nnot}\varphi $
\item[(ii)] $ \circ \varphi_1 \land\circ \varphi_2  \to  \circ(\varphi_1 \# \varphi_2),  \quad \mbox{ for } \# \in \{\land, \lor, \ast, \to\}$
\end{itemize}
are theorems of $\mathcal{PQN}^{\circ\lor}_{\mathsf{m}}$. But again, taking into account that $\mathcal{PQN}^{\circ\lor}_{\mathsf{m}}$ is a particular case of the fuzzy logics  L$_\circ^{\max}$ considered in \cite{CEG14}, for L being a (truth-preserving) fuzzy logic not satisfying the axiom ``$\nnot (\varphi \land \nnot \varphi)$'', Proposition 5.2 in \cite{CEG14} proves that (i) and (ii) indeed hold.  

\begin{proposition}(c.f.  \cite[Prop. 5.2]{CEG14}) The consistency operator $\circ$ in the paraconsistent logic $(\mathcal{PQN}^{\circ\lor}_{\mathsf{m}})^\leq$ satisfies the propagation property with respect to the rest of connnectives, that is, the following conditions hold:  
\begin{itemize}
\item[(i)] $\vdash_{\mathcal{PQN}^{\circ\lor}_{\mathsf{m}}}  \circ \varphi \to \circ {\nnot}\varphi $
\item[(ii)] $\vdash_{\mathcal{PQN}^{\circ\lor}_{\mathsf{m}}}   \circ \varphi_1 \land\circ \varphi_2  \to  \circ(\varphi_1 \# \varphi_2),  \quad \mbox{ for } \# \in \{\land, \lor, \ast, \to\}$
\end{itemize}
\end{proposition}

If $\mathcal{L}$ is an LFI (resp.,~LFU), another desirable property is to recover classical logic reasoning inside L by means of the consistency (resp.,~undeterminedness) operator.   
In general, if a recovery operator $\diamond$ enjoys the propagation property with respect to the classical connectives, say $\land, \lor, \nnot$, in a given LFI or LFU logic $\mathcal{L}$,   then $\mathcal{L}$ satisfies the so-called {\em Derivability Adjustment Theorem} (DAT) when for every finite set of formulas $\Gamma \cup \{\varphi\}$ in the language of classical propositional logic ($\mathcal{CPL}$)

(DAT) $\Gamma \vdash_{\mathcal{CPL}} \varphi$  iff $\Gamma \cup \{\diamond p_1,\dots, \diamond p_n\} \vdash_L \varphi$,

\noindent where $\{p_1,\ldots,p_n\}$ is the set of propositional variables occurring in $\Gamma \cup \{\varphi\}$.

It turns out that, due to the specific properties of the dual operator $\bu^\delta$ in $\mathsf{PQN}$-chains, the LFU logic $\mathcal{PQN}^{\bu}_{\mathsf{Bm}}$ does satisfy the DAT theorem. In fact $\bu^\delta$ acts as a double recovery operator since it recovers both the explosion principle and the excluded-middle principle.  The proof is analogous to that of \cite[Prop. 6.3]{CEG14} since the formula $\bu^\delta \varphi \to \varphi \lor \nnot \varphi$ is valid in $\mathcal{PQN}^{\bu}_{\mathsf{Bm}}$. 

\begin{proposition} \label{noDAT} $\mathcal{PQN}^{\bu}_{\mathsf{Bm}}$ satisfies the following DAT theorem: for any finite set of formulas $\Gamma \cup \{\varphi\}$ in the language of $\mathcal{CPL}$ and on variables $p_1,\ldots, p_n$, 
\begin{center}
$\Gamma \vdash_{\mathcal{CPL}} \varphi$ \quad iff \quad $\Gamma \cup \{\bu^\delta p_1,\dots, \bu^\delta p_n\} \vdash_{\mathcal{PQN}^{\bu}_{\mathsf{Bm}}} \varphi$.
\end{center}
\end{proposition}
 
In contrast, the LFI logic 
$(\mathcal{PQN}^{\circ\lor}_{\mathsf{m}})^\leq$ does not satisfy the above form of the DAT theorem. It is easy to check that, while  $\varphi \lor \nnot \varphi$ is clearly a $\mathcal{CPL}$ theorem, $\varphi \lor \nnot \varphi$ does not follow from $\circ \varphi$ in the logic  $(\mathcal{PQN}^{\circ\lor}_{\mathsf{m}})^\leq$, or equivalently, $\circ \varphi \to (\varphi \lor \nnot \varphi)$ is not a theorem in $\mathcal{PQN}^{\circ\lor}_{\mathsf{m}}$. By completeness, it is enough to show that $\circ \varphi \to (\varphi \lor \nnot \varphi)$ is not a valid formula in all $\mathsf{PQN}^{\circ\lor}_{\mathsf{m}}$-chains. Indeed, as shown also in \cite{CEG14}, let $\bf A$ be a $\mathsf{PQN}^{\circ\lor}_{\mathsf{m}}$-chain such that $\{0,1\} \subsetneq E({\bf A})$, and let $e$ be an $\bf A$-evaluation such that, for some propositional variable $p$, $e(p) \in E({\bf A}) \setminus \{0,1\}$, i.e. $e(\nnot p) = 0$. Then $e(\circ p) = 1$, while $e(p \lor \nnot p) = e(p) < 1$. 

Finally let us analyse the behaviour of the non-falsity preserving LFI logic ${\sf nf}\mbox{-}\mathcal{PQN}^{\circ\lor}_{\mathsf{m}}$ with the respect to the propagation properties and the DAT theorem.

\begin{proposition} \label{prop-ball} The consistency operator $\circ$ in the paraconsistent logic ${\sf nf\mbox{-}}\mathcal{PQN}^{\circ\lor}_{\mathsf{m}}$ satisfies the propagation property with respect to the rest of connectives, that is, the following conditions hold:  
\begin{itemize}
\item[(i)] $ \circ \varphi  \vdash_{{\sf nf}-\mathcal{PQN}^{\circ\lor}_{\mathsf{m}}}  \circ {\nnot}\varphi $
\item[(ii)] $ \circ \varphi_1, \circ \varphi_2 \vdash_{{\sf nf}-\mathcal{PQN}^{\circ\lor}_{\mathsf{m}}}   \circ(\varphi_1 \# \varphi_2),  \quad \mbox{ for } \# \in \{\land, \lor, \ast, \to\}$
\end{itemize}
\end{proposition}

\begin{proof} Thanks to the completeness of ${\sf nf}$-$\mathcal{PQN}^{\circ\lor}_{\mathsf{m}}$, we can prove both properties semantically. 

To prove (i) amounts to checking that, for any evaluation $e$ on a $\PQNA^{\circ\lor}_{\mathsf{m}}$-chain $\bf A$, $e(\circ \varphi) > 0$ then  $e(\circ \neg\varphi) > 0$ as well, or equivalently, due to the particular form of $\circ$ in a chain, if $e(\circ \varphi) = 1$, then  $e(\circ \neg\varphi) =1$. If $e(\circ \varphi) = 1$, it means that $e(\varphi) \in E({\bf A}) = [a, 1]$. Then, if $e(\varphi) = 0$, then $e(\neg\varphi) = e(\circ \varphi) =  e(\circ \neg \varphi) = 1$. Finally, if $e(\varphi) \in E(A) \setminus \{0\}$, then $e(\neg \varphi) = 0$, $e(\circ \varphi) = 1$ and hence  $e(\circ\neg \varphi) = 1$ as well. This ends the proof of (i). 

As for (ii), we prove for the case $\# = \ast$. Again, for a given evaluation $e$ on a $\mathsf{PQN}^{\circ\lor}_{\mathsf{m}}$-chain $A$, assume $e(\circ \varphi_1) > 0$ and $e(\circ \varphi_2) > 0$, that is,  $e(\circ \varphi_1) = e(\circ \varphi_2) = 1$, and let us check that $e(\circ(\varphi_1 \ast \varphi_2)) = 1$ as well. By definition,  $e(\varphi_1 \ast \varphi_2) = \min(e(\varphi_1), e(\varphi_2))$ if $e(\varphi_1) > \neg e(\varphi_2)$, $e(\varphi_1 \ast \varphi_2) = 0$ otherwise. Hence, by observing that, trivially, $\min(e(\varphi_1), e(\varphi_2)) \geq a$ whenever $e(\varphi) \geq a$ and $e(\psi) \geq a$,  we are done. 
Finally, as for the case $\# =\; \to$, assume $e(\circ \varphi_1) > 0$ and $e(\circ \varphi_2) > 0$ (hence $e( \varphi_1),  e(\varphi_2) \geq a$) and let us check that $e(\circ(\varphi_1 \to \varphi_2)) > 0$. Notice that in any  $\mathsf{PQN}$-chain, $x \to y = 1$ if $x \leq y$ and $x \to y = \max(\nnot x, y)$ otherwise. Then, it is easy to check that if $x, y \geq a$ and $x > y$, then  $\max(\nnot x, y) = \max(0, y) = y \geq a$, hence $x \to y \geq a$ in any case, and thus $\circ(x \to y) = 1$. 
\end{proof}

For our final result, we need a preliminary lemma that shows that every theorem of $\mathcal{CPL}$ is indeed a valid formula in the non-falsity preserving variant of $\mathcal{PQN}$.

\begin{lemma} \label{CPL} Let $\varphi$ a formula in the language of $\mathcal{CPL}$. Then if $\varphi$ is a theorem of $\mathcal{CPL}$, then $\models^{(0}_{\mathcal{PQN}}\varphi$.
\end{lemma}

\begin{proof}
Suppose $\not\models^{(0}_{\mathcal{PQN}} \varphi$. Then there exists an evaluation on a $\mathsf{PQN}$-chain $\bf A$ on $[0, 1]$ such that $e(\varphi) = 0$. Let $a \in (0, 1)$ such that $a = \max\{ x \in [0, 1] \mid  \neg x \geq x\}$. Then we define a mapping $h: [0, 1] \to \{0, 1\}$  as follows: 
$$h(x) = \left \{
\begin{array}{ll}
1, & \mbox{if } x > a \\
0, & \mbox{otherwise},
\end{array}
\right .$$
and let the Boolean evaluation $e'$ be defined as $e'(p) = h(e(p))$ for every propositional variable $p$. It is not difficult to check that, for any formula $\psi$ in the language of $\mathcal{CPL}$, $e'(\psi) = h(e(\psi))$. Therefore, $e'(\varphi) = 0$, and hence $\varphi$ is not a theorem of $\mathcal{CPL}$.
\end{proof}

Now we can show that, while $(\mathcal{PQN}^{\circ\lor}_{\mathsf{m}})^\leq$ does not satisfy a DAT theorem of the form of Prop. \ref{noDAT}, ${\sf nf\mbox{-}}\mathcal{PQN}^{\circ\lor}_{\mathsf{m}}$ does. 

\begin{proposition} ${\sf nf-}\mathcal{PQN}^{\circ\lor}_{\mathsf{m}}$ satisfies the following DAT theorem: for any finite set of formulas $\Gamma \cup \{\varphi\}$ in the language of $\mathcal{CPL}$ and on variables $p_1,\ldots, p_n$, 
\begin{center}
$\Gamma \vdash_{\mathcal{CPL}} \varphi$ \quad iff \quad $\Gamma \cup \{\circ p_1,\dots, \circ p_n\} \vdash_{{\sf nf-}\mathcal{PQN}^{\circ\lor}_{\mathsf{m}}} \varphi$.
\end{center}
\end{proposition}

\begin{proof} In the following, since $\Gamma$ is finite, we can assume it consists of a single formula $\gamma$. 
\vspace{.1cm}

($\Rightarrow$) Assume $\gamma \vdash_{\mathcal{CPL}} \varphi$. The idea is to transform a proof of $\varphi$ from $\gamma$ in $\mathcal{CPL}$ to a proof of $\varphi$ from $\gamma \cup \{\circ p_1,\dots, \circ p_n\}$ in $\mathcal{PQN}^{\circ\lor}_{\mathsf{m}}$. To do that we first show that the following ``guarded'' form of Modus Ponens is sound (and thus derivable) in $\mathcal{PQN}^{\circ\lor}_{\mathsf{m}}$: 
$$(MP_\circ) \quad \frac{\circ \psi,\circ \chi, \psi, \psi \to \chi}{\chi}
$$
We have to prove that, for any evaluation $e$ on a totally ordered $\PQNA^{\circ}_{\mathsf{m}}$-algebra $\bf A$, if $e(\circ \psi) > 0, e(\circ \chi) > 0, e(\psi) > 0$, and $e(\psi\to \chi) > 0$, then necessarily $e(\chi) > 0$. Suppose $E({\bf A}) = \{0\} \cup [a, 1]$. Then the previous conditions amount to assuming $e(\psi) \geq a$, $e(\chi) \in E({\bf A})$ and $e(\psi \to \chi) > 0$.  
By way of contradiction, assume $e(\chi) = 0$. Then, since $e(\psi) \geq a$, and thus $e(\psi) > e(\chi)$, we would have $e(\psi \to \chi) \;{=}
\max(\neg e(\psi), e(\chi))= \max(0, 0) = 0$, a contradiction. Thus, it must be $e(\chi)> 0$. (Here we have used the fact that in every $\mathsf{WNM}$-chain, and hence in $\mathsf{PQN}$-chain a fortiori, the implication comes defined as $x \to y = 1$ if $x \leq y$ and $x \to y = \max(\neg x, y)$ otherwise, where $\sim$ is the negation of the $\mathsf{WNM}$-chain, see e.g. \cite{esteva2001monoidal}.)

Now, let $\langle \Pi_0 = \gamma,\Pi_1, \ldots, \Pi_m = \varphi\rangle$ be a proof in $\mathcal{CPL}$, where each $\Pi_i$ with $i > 0$ is either an axiom of $\mathcal{CPL}$ or has been deduced by Modus Ponens from two previous $\Pi_j, \Pi_k$ with $j, k < i$. By observing that 

(i) by the above Lemma \ref{CPL}, all axioms of $\mathcal{CPL}$ are theorems of $\mathcal{PQN}^{\circ\lor}_{\mathsf{m}}$ (i.e. they are always evaluated to a positive value), and

(ii) by the above Prop.\ \ref{prop-ball}, if $q_1, \ldots, q_r$ are the propositional variables appearing in a formula $\psi$, then it follows that 
$\circ q_1,\dots, \circ q_r  \vdash_{{\sf nf-}\mathcal{PQN}^{\circ\lor}_{\mathsf{m}}} \psi$, 

\noindent we can then easily transform the proof $\langle \Pi_0 = \gamma,\Pi_1, \ldots, \Pi_m = \varphi\rangle$ in $\mathcal{CPL}$ into a proof in ${\sf nf\mbox{-}}\mathcal{PQN}^{\circ\lor}_{\mathsf{m}}$ by replacing the applications of Modus Ponens by applications of the guarded Modus Pones  rule $(MP_\circ)$. 




\vspace{.1cm}

($\Leftarrow$)
Suppose  $\Gamma \not\vdash_{\mathcal{CPL}} \varphi$. Then there exists a {\bf 2}-evaluation $e: Var \to \{0, 1\}$ such that $e(\Gamma) = 1$ and $e(\varphi) = 0$. If we define $\circ: {\bf 2} \to {\bf 2}$ such that $\circ(1) = \circ(0) = 1$, then ${\bf 2}^\circ = (\bf 2, \circ)$ is a $\mathsf{PQN}^{\circ\lor}_{\mathsf{m}}$-algebra. Extend $e$ to $e_\circ$ on ${\bf 2}^\circ$, then we have $e_\circ (\circ(p_i)) = 1$, $e_\circ(\Gamma) = 1$ and $e_\circ(\varphi) = 0$, hence $\Gamma \cup \{\circ p_1,\dots, \circ p_n\} \not\vdash_{{\sf nf\mbox{-}}\mathcal{PQN}^{\circ\lor}_{\mathsf{m}}} \varphi$.
\end{proof}

\section{Conclusions and future work} \label{conclusions}

In the present paper,
extending 
results reported in~\cite{FiRi24,ConigBook},
we have investigated
recovery operators within 
quasi-Nelson logic, with a particular focus on the prelinear case. 
We have seen that the main 
algebraic and logical results on residuated lattice-based LFIs/LFUs from~\cite{dIRL} can be recovered
in the quasi-Nelson setting, and our approach  led us beyond~\cite{dIRL}
in that we have been dispensed 
the involutivity assumption.
However, the present study
does not apply to general (distributive) residuated lattices,
for we did add to the picture the powerful Nelson equation, which entails
(distributivity and) three-potency. The project of a truly general extension of~\cite{dIRL} will therefore  have to be addressed in future research. 
Namely, we plan to investigate what properties  are still ensured in 
the algebraic setting of  general non-involutive residuated lattices,
for both the operators $\bullet$ and $\circ$, and how they relate to their Boolean and (almost) involutive special cases.
 As a somewhat more narrow but nevertheless potentially interesting direction to pursue in the prelinear 
(hence, distributive) setting, let us mention the study of recovery operators
on weak nilpotent minimum algebras, of which $\mathsf{PQN}$-algebras constitute a proper subvariety (see~\cite{FlaRi}).

Another direction concerns the interaction, on a logical level, between the consistency and the undeterminedness operators,
a topic which in the present paper we have only cursorily considered
in Section~\ref{sec:both}, and then only in the algebraic setting. 

The attentive reader may recall a few other problems, mainly of technical interest, that we have left open in
the preceding sections:
for example,
the question of the identity of the white operators 
$\ce$ and $\cb$ outside the prelinear setting (cf.~Remark~\ref{rem:cecb}). These will be addressed in a forthcoming publication. 
 {
 As mentioned in the introduction,
one may consider defining properties for 
 the recovery operators that are alternative
 to those considered both in the present paper
 and in~\cite{FiRi24,ConigBook}. For instance, for 
an LFU obtained by adding a determinedness operator $\bu$, instead of the principle~\eqref{eq:fgpd}
considered in the Introduction ($\vdash \phi \lor \nnot \phi \lor \bu \phi$), one can alternatively require 
the principles: 
$\bu^\delta \phi \vdash \phi \lor \nnot \phi$ or
$\vdash \bu^\delta \phi \to (\phi \lor \nnot \phi)$. 
Similarly, for a consistency operator $\circ$, instead of the Boolean condition $\vdash \circ \varphi \lor \nnot \circ \varphi$ one could consider the related but weaker postulate $ \vdash \nnot(\circ \varphi \land \nnot \circ \varphi)$. Interestingly, following this path leads to consistency operators that are fully dual to the undeterminedness operators considered in the present paper.}

Lastly, let us recall that one of the prominent features of quasi-Nelson algebras is that they can be represented
via the so-called \emph{twist} construction, which has proved to be remarkably helpful in the study of these non-classical algebras. The main representation theorem states that every quasi-Nelson algebra $\Al$ can be identified with a pair
$(\Al[H], F)$ where $\Al[H]$ is a Heyting algebra 
enriched with a modal operator (known as a \emph{nucleus}) --
in the prelinear case $\Al[H]$ is a G\"odel algebra, i.e.~a prelinear Heyting algebra
--
and $F$ is a lattice filter on $\Al[H]$. 
This correspondence can be phrased as an equivalence of categories, which in turn can be used for
a topological study quasi-Nelson algebras that builds on the well-known Esakia duality 
for (modal) Heyting algebras (see~\cite{rivieccio2021duality}).
Building on ideas from~\cite{dIRL,esteva2021}, the twist representation has been extended in~\cite{FiRi24,ConigBook}
to account for quasi-Nelson algebras expanded with recovery operators. In the simplest case, 
one thus obtains a correspondence between 
quasi-Nelson algebras 
endowed with an undeterminedness operator and pairs $(\Al[H], F)$ as before, where  now $\Al[H]$ is a modal Heyting algebra 
expanded with a unary operation that realizes the dual pseudocomplement. 
As in the preceding case, an equivalence of categories is readily obtained, and we
 anticipate that,
by suitably extending Esakia duality to deal with pseudocomplement operations,
one may build a duality for quasi-Nelson algebras with undeterminedness operators. 
The case of consistency operators is technically a little more involved, but conceptually very similar, and can be dealt with in a similar way. Such a study will be pursued in a future publication~\cite{JaRi26}.


\begin{thebibliography}{10}
\providecommand{\url}[1]{\texttt{#1}}
\providecommand{\urlprefix}{URL }
\providecommand{\doi}[1]{https://doi.org/#1}

\bibitem{Baaz96}
Baaz, M. Infinite-valued G\"odel logics with 0-1-projections and relativizations. In G\"ODEL 96, LNL 6, P. H\'ajek., ed, pp. 23-33. Springer, 1996.

\bibitem{bou2009logics}
Bou, F., Esteva, F., Font, J.M., Gil, {\`A}.J., Godo, L., Torrens, A.,
  Verd{\'u}, V.: Logics preserving degrees of truth from varieties of
  residuated lattices. Journal of Logic and Computation  \textbf{19}(6),
  1031--1069 (2009)

\bibitem{BuGa}
Busaniche, M., Galatos, N., Marcos, M.A.: Twist structures and {N}elson
  conuclei. Studia Logica  \textbf{110},  949--987 (2022)

\bibitem{BuRi}
Busaniche, M., Rivieccio, U.: {N}elson conuclei and nuclei: the twist
  construction beyond involutivity. Studia Logica \textbf{112}, 1123--1161 (2024)

\bibitem{CaCo16}
Carnielli, W.A., Coniglio, M.E.: {\em Paraconsistent Logic: Consistency, Contradiction and Negation}. Logic, Epistemology, and the Unity of Science Series. New York: Springer, 2016.


\bibitem{CaCoMa} Carnielli, W.A., Coniglio, M.E., Marcos, J.:  {Logics of Formal Inconsistency}. In {\em Handbook of Philosophical Logic}, 2nd edn, Vol. 14, D. Gabbay and F. Guenthner(eds.), pp. 1-93. Springer, 2007.

\bibitem{CaCoRo} 
Carnielli W.A., Coniglio M.E., Rodrigues A.: (2020). Recovery operators, paraconsistency and duality. Logic Journal of the IGPL 28 (5): 624-656 (2020)


\bibitem{Taxo}
Carnielli, W.A., Marcos, J.: A taxonomy of {C}-systems. In: {\em Paraconsistency: The Logical Way to the Inconsistent}, W.A. Carnielli, M. Coniglio and I.M D'ottaviano (eds.), 
pp. 1-94, CRC Press (2002)



\bibitem{cegm}
 Cignoli, R., Esteva, F., Godo, L., and   Montagna, F.: On a class of left-continuous t-norms.   Fuzzy Sets Syst. 131(3): 283--296, 2002.


\bibitem{Cintula-Noguera-book} Cintula, P., and Noguera, C.: {\em Logic and Implication. An Introduction to the General Algebraic Study of Non-classical Logics}, vol. 57 of Trends in Logic, Springer, 2021. 

 \bibitem{ErtolaSC}
 Ertola, R., Esteva, F., Flaminio, T., Godo, L., Noguera, C. Paraconsistency properties in degree-preserving fuzzy logics. {\em Soft Computing} 19(3): 531-546, 2015.

\bibitem{CEG14}
 Coniglio, M.E., Esteva, F., and Godo, L.: Logics of formal inconsistency arising from systems of fuzzy logic. Log. J. IGPL 22(6): 880--904, 2014.

\bibitem{EsDo80} Esteva, F., and Domingo, X.: Sobre negociaciones fuertes y d\'ebiles en $[0, 1]$, Stochastica IV (2) (1980) 141-166.

\bibitem{dIRL}
Esteva, F., Figallo-Orellano, A., Flaminio, T., Godo, L.: Logics of formal
  inconsistency based on distributive involutive residuated lattices. Journal
  of Logic and Computation  \textbf{31}(5),  1226--1265 (2021)

\bibitem{esteva2021}
Esteva, F., Figallo-Orellano, A., Flaminio, T., Godo, L.: Some categorical
  equivalences for {N}elson algebras with consistency operators. In: 19th World
  Congress of the International Fuzzy Systems Association (IFSA), 12th
  Conference of the European Society for Fuzzy Logic and Technology (EUSFLAT)
  420--426. Atlantis Press (2021)

\bibitem{IPMU24}  Esteva, F., Gispert, J., Godo, L.: On the paraconsistent companions of involutive fuzzy logics that preserve non-falsity. In: Marie-Jeanne Lesot et al. (Eds.), Information Processing and Management of Uncertainty in Knowledge-Based Systems, Proceedings of the 20th International Conference (IPMU 2024). Springer.

\bibitem{esteva2001monoidal}
Esteva, F., Godo, L.: Monoidal t-norm based logic: towards a logic for
  left-continuous t-norms. Fuzzy sets and systems  \textbf{124}(3),  271--288
  (2001)

\bibitem{FlaRi}
Flaminio, T., Rivieccio, U.: Prelinearity in (quasi-){N}elson logic. Fuzzy Sets
  and Systems  \textbf{445},  66--89 (2022).
  
  \bibitem{ConigBook}
  Flaminio, T., Godo, L., Rivieccio, U.: Recovery Operators in Quasi-Nelson Logic. To appear.

\bibitem{Font-GTV} Font, J.M., Gil, A.,  Torrens, A. and Verd\'u, V.: {On the infinite-valued \L ukasiewicz logic that preserves degrees of truth}. {\em Archive for Mathematical Logic}, 45, 839--868 (2006).


\bibitem{GaJiKoOn07}
Galatos, N., Jipsen, P., Kowalski, T., Ono, H.: {\em Residuated Lattices: an
  algebraic glimpse at substructural logics}, Studies in Logic and the
  Foundations of Mathematics, vol.~151. Elsevier, Amsterdam (2007)

\bibitem{GEGC}  Gispert, J., Esteva, F.,  Godo, L.,  and  Coniglio, M.E.: On Nilpotent Minimum logics defined by lattice filters and their paraconsistent non-falsity preserving companions. Logic Journal of the IGPL. Forthcoming

\bibitem{GrMMR24} Greati, V.,  Marcelino, S., Marcos, J., Rivieccio, U.: Adding an Implication to Logics of Perfect Paradefinite Algebras. Mathematical Structures in Computer Science  \textbf{34}: 1138-1183 (2024) DOI:10.1017/S0960129524000227.

\bibitem{Hoh-95} H\"ohle, U.: {Commutative, residuated $\ell$-monoids}, In {\em Non-classical Logics and their Applications to Fuzzy Subsets}, U. H\"ohle and E. P. Klement (eds.), pp. 53--106. Kluwer Academic Publishers, 1995.
  
\bibitem{Horcik} Hor\v{c}\'ik, R.: Algebraic Semantics: Semilinear FL-algebras. Chapter 4 in  {\em Handbook of Mathematical Fuzzy Logic-Volume 1}, vol. 37 of Studies in Logic, Mathematical Logic and Foundations, P. Cintula, P. H\'ajek and C. Noguera (eds.), pp. 283-354, London, 2011.

\bibitem{JaRi26} Jansana, R., Rivieccio, U.: Quasi-{N}elson algebras and recovery operators:
a two-sorted duality. Submitted.





\bibitem{Ma05}
Marcos, J.: Nearly every normal modal logic is paranormal. Logique et Analyse  \textbf{48},
  189/192: 279--300 (2005)


\bibitem{thiago2021negation}
Nascimento, T., Rivieccio, U.: Negation and implication in quasi-{N}elson
  logic. Logical Investigations  \textbf{27}(1),  107--123 (2021)

\bibitem{Nels49}
{N}elson, D.: Constructible falsity. Journal of Symbolic Logic  \textbf{14},
  16--26 (1949)

\bibitem{RiSp19}
{Rivieccio}, U., {Spinks}, M.: Quasi-{N}elson algebras. Electronic Notes in
  Theoretical Computer Science  \textbf{344},  169--188 (2019)

\bibitem{RiSp20}
Rivieccio, U., Spinks, M.: {Q}uasi-{N}elson; or, non-involutive {N}elson
  algebras. In: D.~Fazio, A.~Ledda, F.P. (eds.) {\em Algebraic Perspectives on
  Substructural Logics}, Trends in Logic, vol.~55, pp. 133--168. Springer (2020)

\bibitem{riviecciotwonegs}
Rivieccio, U.: Fragments of quasi-{N}elson: {T}wo negations. Journal of Applied
  Logic  \textbf{7},  499--559 (2020)

\bibitem{rivieccio2021fragments}
Rivieccio, U.: Fragments of quasi-{N}elson: The algebraizable core. Logic
  Journal of the IGPL  (2021)

\bibitem{QN4}
Rivieccio, U.: Quasi-{N}4-lattices. Soft Computing  \textbf{26}(6),  2671--2688
  (2022)

\bibitem{FiRi24} 
Rivieccio, U., Figallo, A.: (In)consistency operators on quasi-Nelson algebras. In G.~Metcalfe, T.~Studer, Ruy~de Queiroz (Eds.), Logic, Language, Information, and Computation, 30th International Workshop, WoLLIC 2024 Bern, Switzerland, June 10-13, 2024 Proceedings, pp. 175-192, 2024.

  \bibitem{rfn}
Rivieccio, U., Flaminio, T., Nascimento, T.: On the representation of (weak) nilpotent minimum algebras.  2020 IEEE International Conference on Fuzzy Systems (FUZZ-IEEE), Glasgow, United Kingdom, 2020, pp. 1--8.

\bibitem{rivieccio2021duality}
Rivieccio, U., Jung, A.: A duality for two-sorted lattices. Soft Computing
  \textbf{25}(2),  851--868 (2021)

\bibitem{Sanka}
Sankappanavar, H.: Heyting algebras with dual pseudocomplementation. Pacific
  Journal of Mathematics  \textbf{117}(2) (1985)

\end{thebibliography}
\end{document}